\documentclass[reqno, 12pt]{amsart}

\usepackage{srcltx}

\usepackage{a4wide}

\usepackage{amsfonts, amsthm, amsmath, amssymb}

\RequirePackage{mathrsfs} \let\mathcal\mathscr

\numberwithin{equation}{section}

\newtheorem{theorem}{Th\'eor\`eme}[section]
\newtheorem{lemma}[theorem]{Lemme}
\newtheorem{cor}[theorem]{Corollaire}

\theoremstyle{definition}

\newtheorem*{rem}{Remarque}
\newtheorem*{rems}{Remarques}

\renewcommand{\phi}{\varphi}

\newcommand{\ZZ}{\mathbb{Z}}

\newcommand{\NN}{\mathbb{N}}

\newcommand{\RR}{\mathbb{R}}

\renewcommand{\leq}{\leqslant}
\renewcommand{\le}{\leqslant}
\renewcommand{\geq}{\geqslant}
\renewcommand{\ge}{\geqslant}

\renewcommand{\a}{\mathbf{a}}

\renewcommand{\d}{{\rm d}}
 \newcommand{\e}{{\rm e}}

\DeclareMathOperator{\Mod}{mod}
\renewcommand{\bmod}[1]{\,(\Mod{#1})}
\renewcommand{\rho}{\varrho}

\begin{document}
\title[
Entiers friables dans des progressions arithm\'etiques de grand module
]{Entiers friables dans des progressions arithm\'etiques de grand module}

\author{R.\ de la Bret\`eche}
\address{
Institut de Math\'ematiques de Jussieu-Paris Rive Gauche\\
Universit\'e  Paris Diderot\\
Sorbonne Paris Cit\'e, UMR 7586\\
Case Postale 7012\\
F-75251 Paris CEDEX 13\\ France}
\email{regis.delabreteche@imj-prg.fr}

\author{D. Fiorilli}
\address{ 
Institut de Math\'ematiques de Jussieu-Paris Rive Gauche, UMR 7586 \\ Case Postale 7012\\
F-75251 Paris CEDEX 13\\ France} 
\email{daniel.fiorilli@uottawa.ca}

\maketitle

\begin{abstract}
We study the average error term in the usual approximation to the number of $y$-friable integers congruent to $a$ modulo $q$, where $a\neq 0$ is a fixed integer. We show that in the range $\exp \{ (\log\log x)^{5/3+\varepsilon}\} \leq y \leq x$ and on average over  $q\leq x/M$ with  $M\rightarrow \infty$ of moderate size, this average error term is asymptotic to $-|a|\Psi(x/|a|,y)/2x$. Previous results of this sort were obtained by the second author for reasonably dense sequences, however the sequence of $y$-friable integers studied in the current paper is thin, and required the use of different techniques, which are specific to friable integers.

\end{abstract}

\section{Introduction}

Quoique tr\`es uniforme lorsque \'etudi\'ee macroscopiquement, la r\'epartition des nombres premiers dans les progressions arithm\'etiques poss\`ede plusieurs irr\'egularit\'es. Un r\'ecent exemple est la r\'epartition des nombres premiers congrus \`a $a$ modulo $q$, en moyenne sur $q$. Plus pr\'ecis\'ement, en d\'efinissant
$$ \psi^*(x;q,a):=\sum_{\substack{ p^k\leq x \\ p^k\equiv a \bmod q \\ p^k \neq a }} \log p, $$ 
le deuxi\`eme auteur \cite{F12} a r\'ecemment d\'emontr\'e que la moyenne de la diff\'erence $\psi^*(x;q,a)-x/\phi(q)$ sur les modules $q\leq Q$  premiers avec $a$ d\'epend fortement de la nature arithm\'etique de l'entier $a$. Ce ph\'enom\`ene se quantifie de la mani\`ere suivante. 

\begin{theorem}[{\cite[Theorem 1.1]{F12}}]
\label{thfi}
Soient un entier $a\neq 0$, $\varepsilon>0$ et $B>0$ un nombre r\'eel. Sous les conditions $x\geq 2$, $1 \leq M \leq (\log x)^B,$ nous avons l'estimation
$$  \frac 1{\frac{\phi(a)}a \frac xM}\sum_{\substack{q\leq \frac xM \\ (q,a)=1}}  \left( \psi^*(x;q,a) - \frac x{\phi(q)} \right) = \mu(a,M) + O_{a,\varepsilon, B}\left( \frac 1{M^{\frac{205}{538}-\varepsilon}}\right),$$
 o\`u
 $$ \mu(a,M):= 
 \begin{cases} 
-\frac 12 \log M -C  &\text{ si } a=\pm 1; \\
-\frac 12 \log p  &\text{ si } a=\pm p^k; \\
 0 &\text{ sinon,}
 \end{cases}$$
 et $C$ est une constante absolue explicite.
\end{theorem}

Il se trouve que de telles irr\'egularit\'es existent aussi pour d'autres suites arithm\'etiques $\mathcal A=\{\a(n)\}$, comme par exemple les entiers pouvant s'\'ecrire comme somme de deux carr\'es, ou les entiers cribl\'es \cite{F13}. Une condition importante pour utiliser les techniques de \cite{F13} est que la suite $\mathcal A$ soit assez dense dans $\mathbb N$. De plus, on doit avoir des r\'esultats d'\'equidistribution de $\mathcal A$ dans toutes les progressions arithm\'etiques o\`u la suite $\mathcal A$ est support\'ee (pas n\'ecessairement seulement les progressions $a\bmod q$ avec $(a,q)=1$). 

Le but de notre article est d'\'etudier une suite clairsem\'ee, pour laquelle les techniques de~\cite{F13} ne s'appliquent pas. La suite la plus naturelle et la plus int\'eressante \`a \'etudier est la suite caract\'eristique  des entiers $y$-friables, compte tenu de toute la th\'eorie et des techniques qui ont \'et\'e d\'evelopp\'ees pour son \'etude. En particulier, de puissants r\'esultats d'\'equir\'epartition dans les progressions arithm\'etiques avec $(a,q)=1$ ont \'et\'e obtenus r\'ecemment par Soundararajan \cite{S08}, Harper \cite{H12a}, \cite{H12} et Drappeau \cite{D13}. Il semble donc appropri\'e de d\'emontrer un analogue du Th\'eor\`eme \ref{thfi} pour la suite des entiers friables ; c'est le contenu de notre r\'esultat principal, le Th\'eor\`eme \ref{th}. Avant d'\'enoncer ce r\'esultat, nous allons introduire quelques notations et pr\'esenter le th\'eor\`eme de Drappeau.

Soit $P(n)$ le plus grand facteur premier d'un entier $n$ avec la convention $P(1)=1$. On dit qu'un entier est $y$-friable si $P(n)\leq y$.  
Nous consid\'erons 
\begin{align*}
\Psi(x,y;a,q)&:=\sum_{\substack{n\in S(x,y)\\n\equiv a\bmod q}}1,\qquad 
\Psi^*(x,y;a,q):=\Psi(x,y;a,q)-{\bf 1}_{S(x,y)}(a),
\cr \Psi_q(x,y )&:=\sum_{\substack{n\in S(x,y)\\(n,q)=1}}1,\qquad \Psi(x,y):=\Psi_1(x,y).
\end{align*} 
Nos r\'esultats concernent la moyenne sur $q$ de  
$$E^*(x,y;a,q):=\Psi^*(x,y;a,q)-\frac{\Psi_q(x,y)}{ \phi(q)}$$
lorsque $(a,q)=1$.
Dans toute la suite, nous utilisons les notations usuelles 
$$u:=\frac{\log x}{\log y},\qquad H(u):=\exp\Big\{ \frac{u}{(\log (u+1))^2}\Big\},\qquad L_\varepsilon(y):=\exp\{ (\log y)^{3/5-\varepsilon}\}.$$

\`A la suite de Fouvry--Tenenbaum \cite{FT96} et Harper \cite{H12}, Drappeau montre dans \cite{D13} un r\'esultat de type  Fouvry--Iwaniec pour les entiers friables permettant de consid\'erer des modules $q$ plus grands que $\sqrt{\Psi(x,y)}$. Son r\'esultat est le suivant.

\begin{theorem}[{\cite{D13}}]\label{thsd}
Soient $\varepsilon>0$ et $A>0$. Il existe des constantes absolues $C,\delta>0$ telles que, lorsque $(\log x)^C\leq y\leq x^{1/C}$, nous ayons
\begin{equation}\label{sd1}\sum_{ \substack{q\leq x^{3/5-\varepsilon}\\ (q,a)=1}}\!\!\!\tau(q)^2\max_{z\leq x}| E^*(z,y;a,q)| 
\ll_A \Psi(x,y)\big\{ H(u)^{-\delta}(\log x)^{-A} +y^{-\delta}\big\} \qquad (|a|\leq x^{\delta}),
\end{equation}
\begin{equation}\label{sd2}\sum_{ \substack{q\leq x^{6/11-\varepsilon}\\ (q,a)=1}}\!\!\!\!\!\tau(q)^2\max_{z\leq x} | E^*(z,y;a,q)| \ll_A \Psi(x,y)\big\{ H(u)^{-\delta}(\log x)^{-A} +y^{-\delta}\big\} \quad (|a|\leq x^{1-\varepsilon}).\end{equation} 
\end{theorem}

\noindent{\it Remarque.} Le r\'esultat de Drappeau a \'et\'e \'enonc\'e sans le facteur $\tau(q)^2$ ni le maximum sur $z\leq x$,  mais il est ais\'e de voir que sa d\'emonstration fournit un r\'esultat aussi fort. 
\medskip

Nous d\'efinissons le domaine $(H_\varepsilon)$ en $(x,y)$ par 
$$x\geq 3,\qquad \exp\{ (\log\log  x)^{5/3+\varepsilon}\big\}\leq y\leq x.$$
Hildebrand \cite{H86} a montr\'e la validit\'e de la formule 
\begin{equation}\label{estHild}\Psi(x,y)=x\rho(u)\Big\{ 1+O\Big(\frac{\log (u+1)}{\log y}\Big)\Big\}
\end{equation}
uniform\'ement, pour chaque $\varepsilon>0$, dans le domaine $(H_\varepsilon)$ o\`u $\rho$ d\'esigne la fonction de Dickman qui est continue en $u=1$ d\'erivable sur $]1,\infty[$, valant $1$ sur $]0,1]$ et  solution de l'\'equation diff\'erentielle aux diff\'erences 
$$u\rho'(u)=-\rho(u-1).$$

Lorsque $M\geq 1$ et $a\in \ZZ\smallsetminus \{ 0\}$, nous \'etudions
\begin{equation}\label{defsigma}\sigma(x,y,M; a):=\sum_{ \substack{q\leq x/M\\ (q,a)=1}}E^*(x,y;a,q).\end{equation} 
 
Nous notons $\tau(a)=\tau_2(a)$ le nombre de diviseurs de $a$, $\tau_{k+1}(a)=(1*\tau_k)(a)$ lorsque $k\geq 2$ et $\phi$ la fonction indicatrice d'Euler.

Notre r\'esultat principal est le suivant.
\begin{theorem} \label{th}
Soient $\varepsilon\in\, ]0,1/2[$ et $A>0$. Il existe des constantes absolues $C,\delta>0$ telles  que,    lorsque $(x,y)$ satisfait $(H_\varepsilon)$ et $  y\leq x^{1/C}$,  $1\leq  |a|\leq M^{1/2-\varepsilon}$, et
\begin{equation}\label{inegM}M\leq \min\big\{H(u)^{ \delta} (\log x)^{ A} ,y^{ \delta} \big\},\end{equation}  nous ayons 
\begin{equation}\label{estsigmaM1}\begin{split}\sigma (x,y,M; a) =& -\frac{x\phi(|a|)}{2M |a|}\rho(u_a) +O_{\varepsilon,A}\bigg( \frac{\tau_3(a)^2\Psi(x,y)}{ M L_\varepsilon(M) }  \bigg),\end{split}\end{equation}
avec $u_a:=u-(\log |a|)/(\log y).$ 
\end{theorem}

\begin{rems} 
1-- La condition $ M\leq y $ permet de simplifier les calculs. En effet,   tous les entiers $\leq M$ sont $y$-friables.  
De plus, la pr\'esence de $O(\tau_3(a)\Psi(x,y) y^{-\delta})$ dans le terme d'erreur de~\eqref{estsigmaM1} limite le domaine dans lequel le premier terme du membre de droite de~\eqref{estsigmaM1}  est un terme principal \`a des param\`etres $M$ tels que $M\leq y^{\delta}$.
Pour avoir un \'equivalent, nous devons aussi avoir
$M\leq H(u)^{ \delta} (\log x)^{ A}$. Cette contrainte est impos\'ee par le th\'eor\`eme de Drappeau.

2-- Lorsque  $(x,y)\in (H_\varepsilon)$, l'estimation \eqref{estHild} fournit 
$$\frac{x}{ |a|}\rho(u_a) =\Psi\Big(\frac{x}{|a|},y\Big)
\Big\{ 1+O\Big(\frac{\log (u+1)}{\log y}\Big)\Big\}.$$

3-- Notons que la restriction $(a,q)=1$ dans \eqref{defsigma} correspond aux conditions d'application du th\'eor\`eme de Drappeau. 
Il est \`a noter que le terme principal attendu pour $\Psi^*(x,y;a,q)$ est 
celui de $\Psi^*(x/d,y;a/d,q/d)$ avec $d=(a,q)$, soit $\Psi_{q/d}(x/d,y)/\phi(q/d)$. Lorsque $(a,q)=d$, nous posons donc
$$E^*(x,y;a,q):=\Psi^*(x,y;a,q)-\frac{\Psi_{q/d}(x/d,y)}{ \phi(q/d)}.
$$
La formule
$$ \sum_{ \substack{q\leq x/M }}E^*(x,y;a,q)
=\sum_{ \substack{d\mid a\\a=da' }} \sigma(x/d,y,M; a') $$ coupl\'ee avec le Th\'eor\`eme \ref{th} permet  d'obtenir l'estimation suivante de la somme de  $E^*(x,y;a,q)$ sans la restriction $(a,q)=1$  
\begin{equation} \begin{split} \sum_{ \substack{q\leq x/M }}&E^*(x,y;a,q)
= - \frac{ x}{2M }\rho(u_a)  +O_{\varepsilon}\bigg( \frac{\tau_4(a)\tau_3(a)\Psi(x,y)}{ M L_\varepsilon(M) }  \bigg).\end{split}\end{equation}

4--
Le cas $  \sigma (x,y,1; a)$ correspond au probl\`eme des diviseurs de Titchmarsh \'etudi\'e par Drappeau \cite{D13}. En effet, par exemple lorsque $a\neq 0$, nous avons
$$ \sum_{ \substack{n\in S(x,y)\\n > a}}\tau(n-a)= \sum_{ \substack{q\leq x  }}E^*(x,y;a,q)+\sum_{q\leq x}\frac{\Psi_{q/(q,a)}(x/(q,a),y)}{\phi(q/(q,a))}+O(|a|).$$ 
 Le terme d'erreur provient du fait que si $a<0$, il peut y avoir des diviseurs dans l'intervalle $[x,x+|a|]$ comptabilis\'es dans le membre de gauche mais pas dans le membre de droite.
Dans le cas $M=1$, notre d\'emonstration  nous ram\`ene \`a une expression qui a \'et\'e \'evalu\'ee lorsque $a=1$ dans \cite{D13}. La g\'en\'eralisation \`a $a$ quelconque rel\`eve des m\^emes m\'ethodes. Nous n'avons pas besoin des estimations d\'evelopp\'ees \`a la section 7 de sorte que notre r\'esultat pourrait \^etre dans ce cas \^etre aussi valables dans un domaine de la forme $(\log x)^C\leq y\leq x^{1/C}$ comme dans  \cite{D13}. Le cas $M=1$ ne n\'ecessite pas de d\'evelopper des estimations \`a l'aide de double int\'egrale comme dans les Lemmes \ref{estsigma2} et \ref{lem2sigma}.

5-- Si l'on rempla\c cait dans la d\'efinition de $E^*(x,y;a,q)$ le terme $\Psi^*(x,y;a,q)$ par la quantit\'e $\Psi(x,y;a,q)$, on aurait \`a rajouter le terme $\sum_{q\leq x/m}{\bf 1}_{S(x,y)}(a)$, qui est beaucoup plus grand que le terme principal obtenu dans le Th\'eor\`eme \ref{th}. Cela justifie a posteriori l'usage de $\Psi^*(x,y;a,q)$. 
\end{rems}

Nous pouvons d\'eduire ais\'ement le r\'esultat asymptotique suivant du Th\'eor\`eme \ref{th}.
\begin{cor}  
Soient $\varepsilon\in\, ]0,\tfrac12[$ et $A>0$. Il existe des constantes $C,\delta>0$ telles  que,    lorsque $(x,y)$ satisfait $(H_\varepsilon)$ et $  y\leq x^{1/C}$, $M$ satisfaisant \eqref{inegM},  $1\leq  |a|\leq M^{1/2-\varepsilon}$,     et $ \tau_3(a)\leq L_{\varepsilon/3}(M),$ 
nous avons lorsque $M$ tend vers l'infini
\begin{equation} \begin{split}\sigma (x,y,M; a) \sim& -\frac{\phi(|a|)x}{2M|a|}\rho(u_a)   \sim -\frac{\phi(|a|) }{2M }\Psi\Big(\frac{x}{|a|},y\Big).
   \end{split}\end{equation} 
\end{cor}

\section{Premi\`ere \'etape de la preuve : l'application du th\'eor\`eme de Drappeau}

Soit $C \geq 15/7$, et posons $M_2:= x^{7/15}$. Pour $1\leq M_1\leq M_2$, la contribution dans la somme $\sigma(x,y,M_1; a)$ des $q\leq x^{8/15}=x/M_2$ peut \^etre major\'ee gr\^ace au Th\'eor\`eme \ref{thsd}.
Ainsi nous avons
\begin{equation}\label{qpetit}
\sum_{ \substack{1\leq q\leq x^{8/15}\\ (q,a)=1}}E^*(x,y;a,q)
\ll\Psi(x,y)\big\{ H(u)^{-\delta}(\log x)^{-A} +y^{-\delta}\big\} .
\end{equation}

Nous avons aussi
$$\sum_{ \substack{x^{8/15}<q\leq x/M_1\\ (q,a)=1}}E^*(x,y;a,q)=
\sum_{ \substack{x/M_2<q\leq x/M_1\\ (q,a)=1}}\Psi^*(x,y;a,q)-\sum_{ \substack{x/M_2<q\leq x/M_1\\ (q,a)=1}} \frac{\Psi_q(x,y)}{ \phi(q)}.$$

Posant 
\begin{equation}\label{defsigma12}\sigma_1(x,y,M ; a) :=\sum_{ \substack{x/M<q\leq x\\ (q,a)=1}}\Psi^*(x,y;a,q),\qquad
\sigma_2(x,y,M; a) :=\sum_{ \substack{x/M<q\leq x\\ (q,a)=1}}\frac{\Psi_q(x,y)}{ \phi(q)},
\end{equation}
nous avons
\begin{equation}\label{diffsigma}
\begin{split}\sum_{ \substack{x/M_2<q\leq x/M_1\\ (q,a)=1}}E^*(x,y;a,q)&= \sigma_1(x,y,M_2;a)-\sigma_1(x,y,M_1; a)\cr&\quad -\sigma_2(x,y,M_2;a)+\sigma_2(x,y,M_1; a).
\end{split} \end{equation} 

Nous montrons une premi\`ere estimation de $\sigma_1(x,y,M; a)$ par $\widetilde\sigma_1(x,y,M; a)$ o\`u
\begin{equation}\label{deftildes}\widetilde\sigma_1(x,y,M; a):=\sum_{\substack{ka'= a }}
\sum_{\substack{\ell_1\mid k\\ (\ell_1, a')=1  }} \mu(\ell_1) \sum_{\substack{1\leq r<M/k\\ (r,a')=1}} \frac{\Psi_{a'r\ell_1}( x/k ,y) -\Psi_{a'r\ell_1}( rx/M ,y )}{\phi(r\ell_1)},
\end{equation}
o\`u $\mu$ d\'esigne la fonction de M\"obius.

\begin{lemma} \label{lemsigma}
Soient $\varepsilon>0$ et $A>0$. Il existe des constantes $C,\delta>0$ telles que,   lorsque $(\log x)^C\leq y\leq x^{1/C}$, $1\leq |a|\leq  \min\{y,M\}$, $M\leq x^{7/15}$, nous avons 
\begin{align*}\sigma_1(x,y,M; a)&=\widetilde\sigma_1(x,y,M; a)
  +O\Big(\tau_3(a)\Psi(x,y)\big\{ H(u)^{-\delta}(\log x)^{-A} +y^{-\delta}\big\}  \Big).\end{align*}
\end{lemma}

\begin{proof}
Nous \'ecrivons les entiers $m$ somm\'es dans $\Psi^*(x,y;a,q)$ apparaissant dans l'expression \eqref{defsigma12} de $\sigma_1(x,y,M;a)$.
Nous \'ecrivons $m=a+qt$, puis $k=(m,a)$. Nous avons $P(k)\leq y$, $k\mid a$.  Il existe des  entiers $n,$ $r$, $a'$,    tels que 
$m=nk$, $a=ka'$, $t=kr$ et $n=a'+qr$ avec $(k,q)=1$, $ (a',n)=1$. 
Nous traitons les conditions $(k,q)=1$ et $(a',n)=1$ par des inversions de M\"obius. La condition $P(k)\leq y$ est forc\'ement satisfaite puisque $k\leq |a|\leq y$. De m\^eme, nous utiliserons librement l'in\'egalit\'e $k\leq M$. La condition $x/M<q\leq x$ \'equivaut \`a $rx/M+a'<r\leq rx+a'$.
 
De simples manipulations et une premi\`ere inversion  de M\"obius fournissent
\begin{align*}\sigma_1(x,y,M; a)& =\sum_{\substack{ka'= a}}\sum_{ \substack{x/M<q\leq x\\ (q,a)=1}}\sum_{\substack{n\in S(x/k,y)\\ n\neq a'\\  (n,a')=1\\ n\equiv a'\bmod{q}}}1
\cr&  
=\sum_{\substack{ka'= a}}
\sum_{\substack{\ell_1\mid k\\ (\ell_1, a')=1}} \mu(\ell_1) \sum_{\substack{1\leq r<(x/k-a')M/x\\(r,a')=1}}\,\,\,
\sum_{\substack{n\in S(\min\{x/k, rx+a'\},y)\\ n\neq a'\\(n,a')=1\\ n \equiv a'\bmod{r\ell_1}\\ n>rx/M+a'}} 1.
\end{align*}
Une seconde inversion  de M\"obius pour traiter la condition $(n,a')=1$ fournit alors
\begin{align*}\sigma_1(x,y,M; a)&
=\sum_{\substack{ka'= a}}
\sum_{\substack{\ell_1\mid k\\ (\ell_1, a')=1\\\ell_2a''=a' }} \mu(\ell_1)\mu(\ell_2) \sum_{\substack{1\leq r<(x/k-a')M/x\\(r,a')=1}}
\sum_{\substack{n'\in S(\min\{x/k, rx+a'\}/\ell_2,y)\\ n'\neq a'' \\ n '\equiv a''\bmod{r\ell_1}\\  n'>rx/M\ell_2+a''}} 1.
\end{align*}
Nous avons utilis\'e ici que $(\ell_2,r\ell_1)\mid (a',r\ell_1)=1$.
Nous  avons $\min\{x/k, rx+a'\}=x/k$ sauf si $k=1$, $r=1$, $a<0$. Dans ce cas, le nombre de $n'$ compt\'es en trop satisfaisant  \`a $ rx+a'<n'\ell_2\leq x$ est au plus $O(a\sum_{k,\ell_2}\tau(k)/k\ell_2 )= O(a^4/\phi(a)^3)$.

Ainsi, posant $R_1(a):=a^4/\phi(a)^3+M\tau(a)a^2/\phi(a)^2$, nous obtenons 
\begin{align*}&\sigma_1(x,y,M; a) 
=\sum_{\substack{ka'= a}}
\sum_{\substack{\ell_1\mid k\\ (\ell_1, a')=1\\\ell_2a''= a' }} \mu(\ell_1)\mu(\ell_2)
\sum_{\substack{1\leq r<(x/k-a')M/x\\(r,a')=1}}
\sum_{\substack{n'\in S( x/k\ell_2 ,y) \\ n'\neq a'' \\n '\equiv a''\bmod{r\ell_1}\\ n'>rx/M\ell_2+a''}} 1+O\big(R_1(a)\big)\\
& 
=\sum_{\substack{ka'= a}}\sum_{\substack{\ell_1\mid k\\ (\ell_1, a')=1\\\ell_2a''= a' }}  \mu(\ell_1)\mu(\ell_2)
\sum_{\substack{1\leq r<M/k-a'M/x\\ (r,a')=1}} 
\cr&\qquad\qquad\big(\Psi^*( x/k\ell_2 ,y;a'',r\ell_1) -\Psi^*( rx/M\ell_2+a'' ,y;a'',r\ell_1) \big)+O\big(R_1(a)\big)
\end{align*}

Dans cette estimation, nous pouvons remplacer le cardinal $\Psi^*( rx/M\ell_2+a' ,y;a',r\ell_1)$ par $\Psi^*( rx/M\ell_2 ,y;a',r\ell_1)$
au prix d'un terme d'erreur $O(1+a/kr\ell_1)$ qui, lorsqu'il est somm\'e, fournit une contribution 
$ O( a^2(\log M)/\phi(a)+R_1(a) ).$  
Il vient
\begin{align*} \sigma_1(x,y,M; a) &
=\sum_{\substack{ka'= a}}\sum_{\substack{\ell_1\mid k\\ (\ell_1, a')=1\\\ell_2a''= a' }} \!\!\!\! \mu(\ell_1)\mu(\ell_2)\!\!\!\!\!\!\!
\!\!\!\!\sum_{\substack{1\leq r<M/k-a'M/x\\ (r,a')=1}}\!\!\!\!\!\!\!\!\!\!\!\big(\Psi^*( x/k\ell_2 ,y;a'',r\ell_1) -\Psi^*( rx/M\ell_2  ,y;a'',r\ell_1) \big)
\cr&\quad +O\big(a^2(\log M)/\phi(a)+R_1(a)\big).
\end{align*}

Gr\^ace au Th\'eor\`eme \ref{thsd} de Drappeau, nous pouvons donc approcher $\sigma_1$ par $\sigma_1^*$ avec  
\begin{align*}\sigma_1^*(x,y,M; a)&: 
=\sum_{\substack{ka'= a}}\sum_{\substack{\ell_1\mid k\\ (\ell_1, a')=1\\\ell_2a''= a' }}\mu(\ell_1)\mu(\ell_2)\!\!\!\!\!\!
\sum_{\substack{1\leq r<M/k-a'M/x\\ (r,a')=1}}\!\!\!\!\!\! \frac{\Psi_{r\ell_1}( x/k\ell_2 ,y) -\Psi_{r\ell_1}( rx/M\ell_2 ,y )}{\phi(r\ell_1)}
 \cr
& 
=\sum_{\substack{ka'= a }}\sum_{\substack{\ell_1\mid k\\ (\ell_1, a')=1\\\ell_2a''= a' }} \mu(\ell_1)\mu(\ell_2)\!\!\!\!\!\!
\sum_{\substack{\substack{  n'\in S(x/k\ell_2, y)\\(n', \ell_1)=1}}}\quad   \sum_{\substack{1\leq r<\min\{M/k-a'M/x, n'\ell_2M/x\}\\(r,a'n')=1}} \frac{1}{\phi(r\ell_1)}.
\end{align*} 
Lorsque $a\leq M$, l'erreur commise est 
\begin{align*}
&\ll \sum_{\substack{ka'= a }}\sum_{\substack{ \ell_2\mid  a'  }}
\sum_{\substack{q\leq M\\ (q,a')}}\tau(q)^2\max_{z\leq x}| E^*(z,y;a',q)|
\ll\tau_3(a)\Psi(x,y)\big\{ H(u)^{-\delta}(\log x)^{-A} +y^{-\delta}\big\}.
\end{align*}  

Nous pouvons enlever la sommation en $\ell_2$ en rajoutant la condition $(n,a')=1$.
Il vient
\begin{align*}\sigma_1^*(x,y,M; a)& 
=\sum_{\substack{ka'= a }}\sum_{\substack{\ell_1\mid k\\ (\ell_1, a')=1  }} \mu(\ell_1) 
\sum_{\substack{\substack{  n \in S(x/k , y)\\(n , a'\ell_1)=1}}}\quad   \sum_{\substack{1\leq r<\min\{M/k-a'M/x, nM/x\}\\(r,a'n )=1}} \frac{1}{\phi(r\ell_1)}.
\end{align*}

Dans la somme int\'erieure, on a $\min\{M/k-a'M/x, nM/x\}=nM/x$ sauf pour les entiers $n$ satisfaisant $x/k-a'<n\leq x/k$. Leur contribution est major\'ee par $O(a(\log M)(\log y))$. 
Nous obtenons donc
\begin{align*} \sigma_1^* (x,y,M; a)
&= \widetilde\sigma_1(x,y,M; a)
 +O\big(a(\log M)(\log y)\big),
\end{align*}
puisque
\begin{equation}\label{tildesigma1exp2}\widetilde\sigma_1(x,y,M; a)=\sum_{\substack{ka'= a }}\sum_{\substack{\ell_1\mid k\\ (\ell_1, a')=1  }} \mu(\ell_1) 
\sum_{\substack{  n\in S(x/k, y)\\ n>x/M\\(n ,a' \ell_1)=1}}\quad   \sum_{\substack{1\leq r<  nM/x \\(r, a'n)=1}} \frac{1}{\phi(r\ell_1)}.
\end{equation}
Il reste \`a v\'erifier que la somme $a(\log M)(\log y)+R_1(a)$ est bien englob\'ee dans le terme d'erreur du Lemme \ref{lemsigma} compte-tenu de la taille de $a$ et de $M$.
\end{proof}

\section{Rappel de certains r\'esultats}
\subsection{Rappel concernant les nombres friables}

 Comme il est d'usage, nous uti\-lisons $\alpha$ le point selle de la m\'ethode du col appliqu\'ee \`a l'estimation de $\Psi(x,y)$ c'est-\`a-dire l'unique solution positive $\alpha=\alpha(x,y)$ d\'efinie par 
$$\sum_{p\leq y} \frac{\log p}{p^\alpha-1}=\log x.$$
Comme on a classiquement
$$\sum_{p\le x}\frac{\log p}{ p-1}=\log x-\gamma+o(1)\qquad (x\to\infty)$$ o\`u  $\gamma$  d\'esigne  la
constante d'Euler, on  a aussi pour une constante convenable $x_0$, \begin{equation}\alpha<1\qquad 
(x>x_0,\,x\ge y\ge
2).\label{alpha<1}\end{equation}

Nous notons la partie $y$-friable de la fonction z\^eta de Riemann
$$\zeta(s,y):=\sum_{P(n)\leq y} \frac{1}{n^s}\qquad (\Re e (s)>0).$$
Nous posons
\begin{equation}\label{defsigmak}\begin{split}\phi_0(s,y)&:=\log \zeta (s,y),\quad 
\phi_k(s,y):=\phi_0^{(k)}(s,y),\cr
\quad \sigma_k&=\sigma_k(\alpha,y):=(-1)^k\phi_k(\alpha,y)\quad (k\in\NN).\end{split}\end{equation} 
Nous ferons usage des relations valables lorsque $y\geq (\log x)^C$, o\`u $C>1$,
\begin{equation}\sigma_k\asymp (\log x)( \log y)^{k-1}\qquad
(k\geq 1),\label{estsigma}\end{equation}
   \'etablies aux lemmes 2 et 4 de \cite{HT86}.
Nous notons $\xi=\xi(u)$ l'unique racine r\'eelle non nulle de l'\'equation 
$$\qquad\qquad\e^\xi=1+u\xi \qquad\qquad (u>0,\quad u\neq 1).$$ Nous introduisons aussi la notation 
\begin{equation}\label{defLY}  Y_{\varepsilon}=\exp\{ (\log y)^{3/2-\varepsilon}\}.\end{equation}

Dans la suite de notre article, nous l'utiliserons l'approximation du  point-selle $\alpha$ suivante 
\begin{equation}
\alpha(x,y)=  1-\frac{\xi (u)}{ \log 
y}+O\Big(\frac{1}{
L_\varepsilon(y)}+\frac{1}{ u(\log y)^2}\Big)
\label{estalpha}\end{equation}
valable dans tout domaine o\`u $ (\log 
x)^{C}\leq y\leq x  $ avec $C>1$ de sorte que lorsque $x$ est suffisamment grand $1-\alpha\leq 2/C$.

 Le c\'el\`ebre th\'eor\`eme d'Hildebrand et Tenenbaum~\cite{HT86},  d\'ecrit une approximation uniforme de $\Psi(x,y)$ en fonction de $\alpha=\alpha(x,y)$ lorsque $x\geq y\geq 2$.

\begin{lemma}[{\cite{HT86}}]\label{estcol}  Lorsque $y\geq \log x$, on a
$$
\Psi(x,y)=\frac{x^{\alpha }\zeta(\alpha ,y)}{ \alpha \sqrt{ 2\pi \sigma_2  }}
\Big\{
1+O\Big(\frac{1}{ u}  \Big)\Big\}.
 $$
\end{lemma}

\begin{rem} Nous utiliserons \`a de nombreuses reprises,  comme cons\'equence de ce r\'esul\-tat et de \eqref{estsigma}, 
la majoration 
\begin{equation}\label{majcol}
 {x^{\alpha }\zeta(\alpha ,y)}\ll\Psi(x,y)(\log y)\sqrt{u},
\end{equation}
valable lorsque $(\log x)^C\leq y\leq x$ et pour une constante $C\geq 1$ fix\'ee.
\end{rem}
 
Nous aurons aussi besoin du r\'esultat suivant \cite[lemmes 11 et 12]{HT86}.
\begin{lemma}[{\cite{HT86}}]\label{estcol2}  Soient $\varepsilon >0$ et $C>1$. Il existe une constante $c_1>0$ telle que, lorsque  $(\log x)^C\leq y\leq x$, et $1/\log y\leq T\leq Y_\varepsilon $, on a 
\begin{align*} 
\frac{1}{2\pi i}\int_{\alpha -iT}^{\alpha +iT}\Big|\zeta(s,y)\frac{x^s}{s}\Big| |\d s|&=\Psi(x,y) 
\Big\{
1+O\Big(\frac{1}{u}+{(\log (T\log y))(\log y)\e^{-c_1u }} \Big)\Big\};  \cr
 \int_{\alpha -iT}^{\alpha +iT}\big|\zeta(s,y) {x^s} \big| |\d s|&\ll 
\Psi(x,y) \Big(1+T(\log y)\e^{-c_1u }\Big) 
 .
\end{align*}
\end{lemma}

\begin{proof}  
Nous \'enon\c cons  une majoration ad\'equate de la quantit\' e  
$\zeta (s,y)/\zeta (\alpha ,y)$ dans un large
domaine en $\tau_1=\Im m s_1$ sous la condition $  y\geq (\log x)^{C}$. D'apr\`es la majoration (4.42) de \cite{BT05}, nous avons, pour chaque 
$\varepsilon>0$ fix\'e,  $s =\alpha +i\tau $, les majorations
\begin{equation}\label{majzetay}
 \frac{\zeta (s ,y)}{ 
\zeta (\alpha ,y)}\ll
\left\{\begin{array}{ll}\exp\big\{ -c_0  {(\tau \log y)^2u} \big\}& \qquad \big(|\tau |\le 
1/\log y\big) \cr\cr
 \exp\Big\{\frac{-c_0(\tau\log y)^2  u}{ (\log (u+1))^2+(\tau\log y)^2}\Big\}
&\qquad  \big(1/\log y<|\tau|\le
Y_\varepsilon\big) \end{array}\right.\end{equation} 
 o\`u  
$ Y_\varepsilon $ a \'et\'e introduit en \eqref{defLY}.

Le lemme 11 de \cite{HT86} \'enonce la premi\`ere estimation pour $T=1/\log y$ avec un terme d'erreur en $O(1/u)$. Il s'agit donc de majorer la contribution des segments $[\alpha-iT,\alpha-i/\log y]$ et $[\alpha+i/\log y,\alpha+iT]$.
En utilisant la deuxi\`eme majoration de \eqref{majzetay}, nous obtenons
\begin{align*}\int_{\alpha +i/\log y}^{\alpha +i\log (u+1)/\log y}\!\!\!\big|\zeta(s,y) {x^s} \big| |\d s|&\ll x^\alpha\zeta (\alpha ,y)\int_{1/\log y}^{ \log (u+1)/\log y}\!\!\!\!\!\!\!
\exp\Big\{\frac{-c_0(\tau\log y)^2  u}{ (\log (u+1))^2+(\tau\log y)^2}\Big\}\d\tau\cr&\ll \frac{x^\alpha\zeta (\alpha ,y)}{ \log y\sqrt{u}}
(\log (u+1))H(u)^{-c_0/2}\ll \Psi(x,y)H(u)^{-c_0/3}.
\end{align*}   
De plus, lorsque $  \log (u+1)/\log y\leq T\leq Y_\varepsilon $, gr\^ace \`a \eqref{majzetay}, nous avons
\begin{align*}\int_{\alpha +i\log (u+1)/\log y}^{\alpha +iT} \big|\zeta(s,y) {x^s} \big| |\d s| &\ll  {x^\alpha\zeta (\alpha ,y)}T\e^{-c_0u/2}\ll \Psi(x,y) T(\log y)\sqrt{u}\e^{-c_0u/2};\cr
 \int_{\alpha +i\log (u+1)/\log y}^{\alpha +iT} \Big|\zeta(s,y) \frac{x^s}{s} \Big| |\d s| &\ll  {x^\alpha\zeta (\alpha ,y)}(\log (T\log y)) \e^{-c_0u/2}
\cr&\ll \Psi(x,y) (\log (T\log y))(\log y)\sqrt{u}\e^{-c_0u/2},
\end{align*}
o\`u nous avons utilis\'e \eqref{majcol}.
En choisissant $c_1<\tfrac12 c_0 $, nous obtenons le r\'esultat.

La troisi\`eme estimation s'obtient   de la m\^eme mani\`ere que  la deuxi\`eme. 
\end{proof}
\subsection{Bornes de la fonction z\^eta}

Nous aurons besoin du lemme suivant.
\begin{lemma}\label{zeta} Soient $\vartheta=\tfrac{32}{205}$ et $\varepsilon>0$. Dans un domaine d\'efini par $|\sigma+i\tau-1|>1/10$, nous avons
$$ \zeta(\sigma+i\tau)\ll_\varepsilon (|t|+1)^{\mu_\zeta(\sigma)+\varepsilon}$$
avec
$$\mu_\zeta(\sigma):=\left\{\begin{array}{ll}1/2-\sigma & { si }\, \sigma \leq 0, \cr1/2+(2\vartheta-1)\sigma & { si }\, 0\leq \sigma \leq 1/2, \cr
 2\vartheta (1-\sigma) & { si }\, 1/2\leq \sigma \leq 1, \cr
0&{ { si }\, \sigma \geq 0. }
\end{array}\right.$$
\end{lemma}

\subsection{Approximation de $\zeta(s,y)$ et la fonction de Dickman }

Nous utiliserons le lemme III.5.16 de \cite{T08}  qui repose sur  la r\'egion sans z\'ero de Korobov--Vinogradov de la fonction z\^eta de Riemann qui s'\'ecrit sous la forme
\begin{equation}\label{KV}\sigma\geq 1-\frac{c}{(\log (|\tau|+3))^{2/3}(\log\log  (|\tau|+3))^{2/3}},\qquad \tau \in \RR. 
\end{equation}
o\`u $c$ est une constante absolue suffisamment petite. Nous rappelons que nous pouvons choisir $c>0$ de sorte que  dans cette r\'egion
\begin{equation}\label{majzeta'/zeta} \qquad\qquad\Big|\frac{1}{\zeta}(s)\Big|+\Big|\frac{\zeta'}{\zeta}(s)\Big|\ll \big(\log | \tau| \big)^{2/3}\big( \log\log |\tau| \big)^{1/3} \qquad\qquad (|\tau|\geq 3). 
\end{equation}

Soit $\widehat f$ la transform\'ee de Laplace d'une fonction $f$ d\'efinie par
$$\widehat f(s):=\int_{0}^\infty \e^{-st}f(t)\d t.$$
Du lemme III.5.16 de \cite{T08} d\'ecoule le r\'esultat suivant. 
\begin{lemma}[{\cite{T08}}]\label{zetasy}Soit  $\varepsilon\in\, ]0,\tfrac12[$.  Il existe un nombre  r\'eel  $y_0=y_0(\varepsilon)$,  tel  que, sous les conditions
$$y\geq y_0(\varepsilon),\qquad \sigma\geq 1-\frac{1}{(\log y)^{2/5+\varepsilon}   },\quad |\tau| \leq L_{\varepsilon}(y),$$
on ait uniform\'ement
$$\zeta(s,y)= \zeta(s)(s-1)(\log y)\widehat\rho((s-1)\log y)\Big\{ 1+O\Big(\frac{1}{L_{\varepsilon}(y)} \Big)\Big\} .$$
\end{lemma}

Nous rassemblons dans le lemme suivant les informations concernant la fonction $\rho$ de Dickman dont nous aurons besoin  (voir le  lemme III.5.12 et le th\'eor\`eme III.5.13 de \cite{T08}).

\begin{lemma}\label{lemmarho}Nous avons  
\begin{equation}\label{majrhou} \widehat\rho(-\xi(u)+i\tau )\ll \left\{\begin{array}{ll}
\sqrt{u}\rho(u)\e^{u\xi(u)}\e^{-u\tau^2/(2\pi^2)} & \hbox{si } |\tau|\leq \pi,\\
\sqrt{u}\rho(u)\e^{u\xi(u)}H(u)^{-1}& \hbox{si } |\tau|> \pi,\end{array}{}\right.\end{equation} 
et lorsque $|\tau|>1+u\xi(u)$ et $s=-\xi(u)+i\tau $
\begin{equation}\label{majtaugrd}\widehat\rho(s )=\frac{1}{s}\Big\{ 1+O\Big(\frac{1+u\xi(u)}{|s|}\Big)\Big\}.
\end{equation}
Cette estimation vaut encore pour $s=(\alpha-1)\log y+i\tau$.\par
Lorsque $u\geq 1$, nous avons 
\begin{equation}\label{estrhou} \rho(u)=\sqrt{\frac{\xi'(u)}{2\pi}}\exp\big\{ \gamma-u\xi(u)+I(\xi(u))\big\}\bigg\{ 1+O\Big(\frac{1}{u}\Big)\bigg\}\end{equation}
avec, lorsque $u$ tend vers l'infini,
\begin{equation}\label{estIxi}I(\xi(u)):=\int_0^{\xi(u) }\frac{\e^t-1}{t}\d t \sim u ,\qquad\xi'(u)\sim 1/u.\end{equation}
\end{lemma}

\section{Calculs de s\'eries de Dirichlet}

Au Lemme \ref{estsigma2}, nous \'ecrirons la somme $\sigma_2(x,y,M; a)$ comme une double int\'egrale complexe \`a partir  de la formule
$$\sigma_2(x,y,M; a) =\sum_{n\in S(x,y)} \sum_{\substack{x/M<q\leq x\\ (q,na)=1}}\frac{1}{\phi(q)}.$$ 
Soit $$\Phi_2(s;a):=\sum_{(n,a)=1}\frac{1}{\phi(n)n^s}.$$
Nous noterons dans toute la suite 
$$g_m(s):=\prod_{p\mid m} \Big(1-\frac{1}{p^s}\Big).$$

Nous sommes amen\'es \`a consid\'erer 
 \begin{equation}\label{defFa}F_a(s_1,s_2;y):=\sum_{P(n)\leq y} \frac{1}{n^{s_1}} \Phi_{2}(s_2;an).
\end{equation}

Le lemme suivant contient le calcul de $F_a(s_1,s_2;y)$.

\begin{lemma}\label{calculFa}  Lorsque $1\leq |a|\leq y$, $\Re e(s_1)>0$, $\Re e(s_2)>1$, nous avons 
$$F_a(s_1,s_2;y)=F_{1} (s_1,s_2;y)\psi_{2}(s_1,s_2;a)$$
avec
\begin{align} F_{1} (s_1,s_2;y) & =
\zeta(s_1,y)    
\prod_{\substack{ p> y}}\Big(1+\frac{1 }{(1-1/p)(p^{s_2+1}-1)}\Big)
\prod_{\substack{ p\leq y}}\Big(1+\frac{(1-1/p^{s_1}) }{(1-1/p)(p^{s_2+1}-1)}\Big);\notag
\\\label{defpsi2}\psi_{2}(s_1,s_2;a)&:=g_a(s_2+1) 
\prod_{\substack{ p \mid a}} \Big(1+\frac{(1-1/p^{s_1-1}) }{(1-1/p) p^{s_2+2} }\Big)^{-1}
.\end{align}
\end{lemma}

\begin{proof}
D'apr\`es la formule (23) de \cite{F12}, un simple calcul fournit
\begin{align*}\Phi_2(s;a)&=\prod_{p\nmid a} \Big(1+\frac{1}{(1-1/p)(p^{s+1}-1)}\Big)\cr&
=\prod_{p\mid a} \Big(1+\frac{1}{(1-1/p)(p^{s+1}-1)}\Big)^{-1}\zeta(s+1)
\prod_{p } \Big(1+\frac{1}{(p-1) p^{s+1} }\Big).
\end{align*} 
Nous pouvons aussi \'ecrire
\begin{equation}\label{calculPhi2a}\Phi_2(s;a)=  {g_a(s+1)}\prod_{p\mid a} \Big( {1+\frac{1}{p^{s+1}(p-1 )}}\Big)^{-1} \zeta(s+1)\zeta(s+2)
\prod_{p } \Big(1+\frac{1-1/p^{s +1}}{(p-1) p^{s+2} }\Big)
.\end{equation}  

Nous obtenons 
\begin{align*} F_a(s_1,s_2;y)&=\Phi_2(s_2;a)
\sum_{P(n)\leq y} \frac{1}{n^{s_1}}  \prod_{\substack{p\mid n\\p\nmid a}} \frac{  1-1/p^{s_2+1}}{ 1+ {1}/{(p-1) p^{s_2+1}}} 
\cr&=\Phi_{2}(s_2;a)\prod_{\substack{p\mid a }}(1-1/p^{s_1})^{-1}
 \prod_{\substack{p\leq y\\ p\nmid a}}\Big\{ 1+ \frac{  1-1/p^{s_2+1}}{(p^{s_1}-1)( 1+ {1}/{(p-1) p^{s_2+1}})} \Big\}
\end{align*}
Or
$$ 1+ \frac{  1-1/p^{s_2+1}}{(p^{s_1}-1)( 1+ {1}/{(p-1) p^{s_2+1}})}=\frac{1}{1-1/p^{s_1}}\Big( 1- \frac{1   }{p^{s_1+s_2+1}(1-1/p+1/ p^{s_2+2})} \Big)$$
donc  
\begin{equation}\label{eqFaF1}\begin{split}F_a&(s_1,s_2;y) =   \Phi_{2}(s_2;a) \zeta(s_1,y)
\prod_{\substack{ p\leq y\\p\nmid a}}\Big( 1- \frac{1   }{p^{s_1+s_2+1}(1-1/p+1/ p^{s_2+2})} \Big) \cr&
=\psi_{2}(s_1,s_2;a)\Phi_{2}(s_2;1) \zeta(s_1,y)
\prod_{\substack{ p\leq y}}\Big( 1- \frac{1   }{p^{s_1+s_2+1}(1-1/p+1/ p^{s_2+2})} \Big).
\end{split}\end{equation}
De la formule
\begin{align*}\Phi_{2}(s_2;1) & 
\prod_{\substack{ p\leq y}}\Big( 1- \frac{1   }{p^{s_1+s_2+1}(1-1/p+1/ p^{s_2+2})} \Big)\cr&=\prod_{\substack{ p> y}}\Big(1+\frac{1 }{(1-1/p)(p^{s_2+1}-1)}\Big)
\prod_{\substack{ p\leq y}}\Big(1+\frac{(1-1/p^{s_1}) }{(1-1/p)(p^{s_2+1}-1)}\Big),
\end{align*}
nous en d\'eduisons la valeur de $F_1(s_1,s_2;y)$ annonc\'ee.
\end{proof}

Nous observons ensuite que, lorsque $\ell\leq y$, on a 
\begin{equation}\label{zetayell}\sum_{\substack{P(n)\leq y\\(n,\ell)=1}} \frac{1}{n^{s}} =g_\ell(s)\zeta(s,y). \end{equation} 
Nous serons amen\'es \`a consid\'erer
\begin{equation}\label{defGa}G_{ a}(s_1,s_2;y) :=
\sum_{\substack{ka'= a }}
\sum_{\substack{\ell_1\mid k\\ (\ell_1, a')=1  }} \frac{\mu(\ell_1)}{\phi( \ell_1)} \sum_{\substack{  (r,a')=1}} \frac{\phi( \ell_1)}{\phi(r\ell_1)(rk)^{s_2}}
\frac{g_{a'r_y\ell_1}(s_1) }{k^{s_1}} \zeta(s_1,y).
\end{equation}
o\`u $r_y$ est la partie $y$-friable de $r$.

\begin{lemma}\label{calculGa} Lorsque $1\leq |a|\leq y$, $\Re e(s_1)>0$, $\Re e(s_2)>1$, nous avons 
$$G_a(s_1,s_2;y)=G_{1} (s_1,s_2;y)\psi_{2}(s_1,s_2;a)K_a( s_1,s_2 ) $$
avec
\begin{align*} G_{1} (s_1,s_2;y) & =F_{1} (s_1,s_2;y),
\cr K_a( s_1,s_2 ) &=\frac{g_a(s_1)}{ g_a(s_1+s_2)}
\prod_{p^\nu\parallel a } 
 \left(    1+\frac{1-1/p^{ s_2 } }{ p^{\nu(s_1+s_2)}} \left(   \frac{1}{ p^{s_1}-1}-\frac{1-1/p^{s_1+s_2}}{(p-1)(1-1/  p^{s_2+1}  )} \right) \right) 
.\end{align*}
De plus, la fonction $H_1 $ d\'efinie par
\begin{equation}\label{G1H1}
G_1(s_1,s_2;y)=\zeta(s_1,y)\frac{\zeta(s_2+1)\zeta(s_2+2)}{\zeta(s_1+s_2+1,y)}   H_1(s_1,s_2;y)
\end{equation}
v\'erifie
$$H_1(s_1,s_2;y) = \prod_{p>y}\Big(1+\frac{ 1-p^{s_2+1}  }{ p^{s_2+3}(1-1/p)} \Big)\prod_{\substack{ p\leq y }}\Big(1-\frac{p^{-s_2-2}+p^{ -s_1} -p^{-s_1-s_2-1}-p^{-1} }{ p^{s_2+2} (1-1/p)(1-1/p^{s_1+s_2+1} ) }\Big) $$
et
$|H_1(s_1,s_2;y)|\ll_\varepsilon 1 $ lorsque $\Re e(s_2)\geq -\tfrac32+\varepsilon$, $\Re e(s_1+s_2)\geq -1+\varepsilon$, $\Re e(s_1+2s_2)\geq -2+\varepsilon$.
\end{lemma}

\begin{rem}
Il est important d'observer que 
$K_a( s_1,0 )=1$. 
\end{rem}

\begin{proof}
La fonction $$r\mapsto \frac{\phi( \ell_1)}{\phi(r\ell_1)}  =\prod_{p^\nu\parallel r} \frac{\phi( \ell_1)}{\phi(p^\nu\ell_1)}$$ est mutiplicative. 
Ainsi
\begin{align*}\sum_{\substack{  (r,a')=1}} \frac{\phi( \ell_1)}{\phi(r\ell_1) r^{s_2}}&
\frac{g_{a'r_y\ell_1}(s_1) }{ g_{a' \ell_1}(s_1) } =
\prod_{\substack{p\nmid a'\ell_1 \\ p\leq y}}\Big(1+\frac{(1-1/p^{s_1}) }{(1-1/p)(p^{s_2+1}-1)}\Big) 
\prod_{\substack{p\mid \ell_1  }}\Big(1+\frac{1 }{ (p^{s_2+1}-1)}\Big) \cr&\qquad\qquad\qquad 
\prod_{\substack{p\nmid a'\ell_1 \\ p>y}}\Big(1+\frac{1 }{(1-1/p)(p^{s_2+1}-1)}\Big) 
\cr
 =&\frac{1}{g_{\ell_1}(s_2+1)}
\prod_{\substack{p\nmid a'\ell_1 }}\Big(1+\frac{(1-1/p^{s_1}) }{(1-1/p)(p^{s_2+1}-1)}\Big) \cr&
\prod_{ p>y} \Big(1+\frac{1 }{(1-1/p)(p^{s_2+1}-1)}\Big)\Big(1+\frac{(1-1/p^{s_1}) }{(1-1/p)(p^{s_2+1}-1)}\Big)^{-1}. 
\end{align*} 
Il vient
\begin{equation}\label{GaG1}\begin{split}G_{ a}&(s_1,s_2;y)\cr& = \zeta(s_1,y) 
\sum_{\substack{ka'= a }}
\sum_{\substack{\ell_1\mid k\\ (\ell_1, a')=1  }} \frac{\mu(\ell_1)g_{a' \ell_1}(s_1) }{\phi( \ell_1) k^{s_1+s_2}g_{\ell_1}(s_2+1)}  
\prod_{\substack{p\nmid a'\ell_1 }}\Big(1+\frac{(1-1/p^{s_1}) }{(1-1/p)(p^{s_2+1}-1)}\Big)\cr&\quad
\prod_{ p>y} \Big(1+\frac{1 }{(1-1/p)(p^{s_2+1}-1)}\Big)\Big(1+\frac{(1-1/p^{s_1}) }{(1-1/p)(p^{s_2+1}-1)}\Big)^{-1}
\cr&= G_{1}(s_1,s_2;y)\psi_{1}(s_1,s_2;a)\end{split}
\end{equation}
o\`u  
$$\psi_{1}(s_1,s_2;a):=\sum_{\substack{ka'= a }}\frac{1}{  k^{s_1+s_2}} \prod_{\substack{ p\mid a' }}\frac{1-1/p^{s_1} } {\Big(1+\frac{(1-1/p^{s_1}) }{(1-1/p)(p^{s_2+1}-1)}\Big)}\!\!\!
\sum_{\substack{\ell_1\mid k\\ (\ell_1, a')=1  }} \frac{\mu(\ell_1)}{\phi( \ell_1)}  \prod_{\substack{ p\mid  \ell_1 }}\frac{  \frac{1-1/p^{s_1}}{1-1/  p^{s_2+1}  }  } {\Big(1+\frac{(1-1/p^{s_1}) }{(1-1/p)(p^{s_2+1}-1)}\Big)}.$$
Notons 
\begin{align*}\prod_{p\leq y}\Big(1+&\frac{(1-1/p^{s_1}) }{(1-1/p)(p^{s_2+1}-1)}\Big) =\zeta(s_2+1,y)\prod_{\substack{ p\leq y }}\Big(1+\frac{ 1- p^{1-s_1}  }{p^{s_2+2}(1-1/p)  }\Big)
\cr&=\zeta(s_2+1,y)\zeta(s_2+2,y)  \prod_{\substack{ p \leq y}}\Big(1-\frac{p^{-s_2-2}+p^{1-s_1} -p^{-s_1-s_2-1}-p^{-1} }{ p^{s_2+2} (1-1/p)  }\Big)
\cr&=\frac{\zeta(s_2+1,y)\zeta(s_2+2,y)}{\zeta(s_1+s_2+1,y)}   \prod_{\substack{ p\leq y }}\Big(1-\frac{p^{-s_2-2}+p^{ -s_1} -p^{-s_1-s_2-1}-p^{-1} }{ p^{s_2+2} (1-1/p)(1-1/p^{s_1+s_2+1} ) }\Big).\end{align*}
D'autre part,
$$\prod_{p>y}\Big( 1+\frac{1 }{(1-1/p)(p^{s_2+1}-1)} \Big)=\prod_{p>y}\frac{ 1+(1-p^{s_2+1})/(p^{s_2+3}(1-1/p)) }{(1-1/p^{s_2+1})(1-1/p^{s_2+2})}.$$
Cela fournit bien la formule \eqref{G1H1}.

En utilisant 
\begin{align*}\prod_{\substack{ p\mid a' }}  &{\Big(1+\frac{(1-1/p^{s_1}) }{(1-1/p)(p^{s_2+1}-1)}\Big)}\prod_{\substack{ p\mid k \\p\nmid a'}}  {\Big(1+\frac{(1-1/p^{s_1}) }{(1-1/p)(p^{s_2+1}-1)}\Big)}\cr&=\prod_{\substack{ p\mid a  }}  {\Big(1+\frac{(1-1/p^{s_1}) }{(1-1/p)(p^{s_2+1}-1)}\Big)},\end{align*}
nous obtenons
\begin{equation} \label{calculpsi1}
\begin{split}\psi_{1}&(s_1,s_2;a)\cr=&\sum_{\substack{ka'= a }}\frac{1}{  k^{s_1+s_2}}\prod_{\substack{ p\mid a' }}\frac{1-1/p^{s_1}  } {\Big(1+\frac{(1-1/p^{s_1}) }{(1-1/p)(p^{s_2+1}-1)}\Big)}
\sum_{\substack{\ell_1\mid k\\ (\ell_1, a')=1  }} \frac{\mu(\ell_1)}{\phi( \ell_1)}  \prod_{\substack{ p\mid  \ell_1 }}\frac{  \frac{1-1/p^{s_1}}{1-1/  p^{s_2+1}  }  } {\Big(1+\frac{(1-1/p^{s_1}) }{(1-1/p)(p^{s_2+1}-1)}\Big)}  \cr
=&\sum_{\substack{ka'= a }}\frac{1}{  k^{s_1+s_2}} \prod_{\substack{ p\mid a' }}\frac{1-1/p^{s_1}  } {\Big(1+\frac{(1-1/p^{s_1}) }{(1-1/p)(p^{s_2+1}-1)}\Big)}
\prod_{\substack{p\mid k\\ p\nmid  a'}} \left( 1-\frac{1}{p-1}  \frac{  \frac{1-1/p^{s_1}}{1-1/  p^{s_2+1}  }  } {\Big(1+\frac{(1-1/p^{s_1}) }{(1-1/p)(p^{s_2+1}-1)}\Big)}\right)
\cr
=& \psi_{2}(s_1,s_2;a)K_a( s_1,s_2 )\end{split}\end{equation} 
avec
\begin{equation}\label{defKa}K_a( s_1,s_2 ) :=g_a(s_1)\sum_{\substack{ka'= a }}\frac{1}{  k^{s_1+s_2}} 
\prod_{\substack{p\mid k\\ p\nmid  a'}}  \left(  \frac{1}{ 1-1/p^{s_1}}+\frac{p^{-s_2 }  -1 }{(p-1)(1-1/  p^{s_2+1}  )} \right) 
  .\end{equation} 
 Un calcul fournit
\begin{equation}\label{estKa}  K_a( s_1,s_2 )  =\frac{g_a(s_1)}{ g_a(s_1+s_2)}
\prod_{p^\nu\parallel a } 
 \left(    1+\frac{1-1/p^{ s_2 } }{ p^{\nu(s_1+s_2)}} \left(   \frac{1}{ p^{s_1}-1}-\frac{1-1/p^{s_1+s_2}}{(p-1)(1-1/  p^{s_2+1}  )} \right) \right) .\end{equation} 
 
 Ainsi
\begin{equation}\label{calculG}G_{a}(s_1,s_2;y)=F_{a}(s_1,s_2;y)K_a( s_1,s_2 ) .\end{equation} 
\end{proof}

Dans le lemme suivant, nous \'enon\c cons les majorations de $\psi_1$ et $\psi_2$ dont nous aurons besoin.

\begin{lemma}\label{majpsi1} Soit $\delta\in \,]0,1[$. Il existe une constante $C>1$ telle que nous ayons lorsque $(\log x)^C\leq y\leq x $,  $1\leq |a|\leq y$,
$\Re e(s_1)=\alpha$, $\Re e (s_2)\geq -(1-\delta),$
\begin{equation}\label{maj1psi1}
\big|\psi_1(s_1,s_2;a)\big|+\big|\psi_2(s_1,s_2;a)\big|\ll_\delta\tau(a).
\end{equation}
De plus, il existe une constante  $C>1$   telle  que nous ayons lorsque $(\log x)^C\leq y\leq x $, $1\leq |a|\leq y$,  $0\leq \beta_1\leq 1/6$, $0\leq \beta_2\leq 1/6$, 
$\Re e(s_1)=\alpha-\beta_1$, $\Re e (s_2)\geq -\alpha-\beta_2,$
\begin{equation}\label{maj2psi1}\psi_{1}(s_1,s_2;a)\ll |a|^{2\beta_1+2\beta_2}\tau_3(a)^2.\end{equation}
\end{lemma}

\begin{proof} Nous choisissons $C$ tel que $\alpha-1+\delta>\tfrac12\delta$. Lorsque $\Re e(s_1)=\alpha$, $\Re e (s_2)\geq -(1-\delta),$
nous avons 
$$\bigg|\frac{1-1/p^{ s_2 } }{ p^{\nu(s_1+s_2)}} \left(   \frac{1}{ p^{s_1}-1}-\frac{1-1/p^{s_1+s_2}}{(p-1)(1-1/  p^{s_2+1}  )}\right)\bigg| 
\leq \frac{2}{p^{2\alpha-2+2\delta}(1-p^{-\delta})}+ \frac{2}{(p-1) (1-p^{-\delta})} 
$$
et donc par \eqref{estKa}
$$K_a( s_1,s_2 )\ll 1/\big(g_a(1)^3 g_a(\delta) g_a(2\alpha-2+2\delta)^2\big)$$
alors que d'apr\`es le Lemme \ref{calculFa}, nous avons
$$\psi_{2}(s_1,s_2;a)\ll_\delta 1/g_a(\delta)\ll_\delta \tau(a), $$ ce qui fournit la  majoration
$$\psi_{1}(s_1,s_2;a)\ll_\delta 1/\big(g_a(1)^3 g_a(\delta)^2 g_a(2\alpha-2+2\delta)^2\big)\ll_\delta\tau(a).$$

 Pla\c cons-nous maintenant dans le cas o\`u  
$\Re e(s_1)=\alpha-\beta_1$, $\Re e (s_2)\geq -\alpha-\beta_2$ avec  $0\leq \beta_i\leq   \tfrac13$.
D'apr\`es \eqref{defpsi2}, nous avons 
$$\psi_{2}(s_1,s_2;a)\ll |g_a(s_2+1)|/|g_a(1-\beta_2)|^2. $$
D'apr\`es \eqref{defKa},
\begin{equation} g_a(s_2+1)K_a( s_1,s_2 )  =g_a(s_1)\sum_{\substack{ka'= a }}\frac{g_{a'}(s_2+1)}{  k^{s_1+s_2}} 
\prod_{\substack{p\mid k\\ p\nmid  a'}}  \left(  \frac{ 1-1/  p^{s_2+1}   }{ 1-1/p^{s_1}}+\frac{p^{-s_2 }  -1 }{(p-1)} \right) 
  .\end{equation}  
Nous avons 
\begin{align*} K_a&( s_1,s_2 ) g_a(s_2+1) \cr&\ll \frac{1}{g_a(\alpha-\beta_1)}
\prod_{p^\nu\parallel a} p^{ (\beta_1+\beta_2) \nu }\Big(\nu p^{-\beta_1-\beta_2  }  (1+p^{\alpha+\beta_2-1})+  \frac{p^{-1}+p^{ \alpha+\beta_2-1}}{1-1/p}+ \frac{1+p^{\alpha+\beta_2-1}}{1-p^{-\alpha+\beta_1}}\Big)
\cr&\ll \frac{|a|^{ \beta_1+\beta_2}}{g_a(\alpha-\beta_1)}
\prod_{p^\nu\parallel a} \Big( \nu (p^{-\beta_1-\beta_2}+p^{\alpha -1-\beta_1})+  \frac{p^{-1}+p^{ \alpha+\beta_2-1}}{1-1/p} +\frac{1+p^{\alpha+\beta_2-1}}{1-p^{-\alpha+\beta_1}}\Big)
\cr&\ll \frac{|a|^{2(\beta_1+\beta_2)}}{g_a(1)g_a(\alpha-\beta_1)^2}
\prod_{p^\nu\parallel a} \big(   2\nu +3+1/p   \big)  \ll  \tau_3(a)\tau(a)|a|^{2(\beta_1+\beta_2)}.
\end{align*} 
Cela fournit la majoration annonc\'ee puisque 
$$ \tau_3(a)\tau(a)/|g_a(1-\beta_2)|^2\ll  \tau_3(a)^2 .$$  
\end{proof}

\section{Estimation de $\widetilde\sigma_1$ et $\sigma_2$ par des int\'egrales doubles}
Nous rappelons la d\'efinition \eqref{deftildes} de $\widetilde\sigma_1(x,y,M; a)$ et \eqref{defsigma12} celle de $\sigma_2(x,y,M; a)$.

\begin{lemma}\label{estsigma2}  Il existe une constante $C>1$ telle que nous ayons lorsque $(\log x)^C\leq y\leq x $, $1\leq |a|\leq y$, $(\log x)^{13}\leq T_1\leq T_2^{1/4}\leq \sqrt{x}$, $ T_2\geq (\log x)^2$, $\kappa=1/\log x$, $M\leq x^{1/2}$, 
\begin{align*} \widetilde\sigma_1  (x,y,M; a)    
&=\frac{1}{(2\pi i)^2}\int_{\alpha -iT_1}^{\alpha +iT_1} \!\!\!\int_{\kappa-iT_2}^{\kappa+iT_2}G_a(s_1,s_2;y)x^{s_1}M^{s_2} \frac{\d s_2\d s_1}{(s_2+s_1)s_2}\cr&\quad+O\Big(\tau(a) \frac{\Psi(x,y)}{ T_1^{1/3}}+\tau(a)\frac{\Psi(x,y)}{ T_2^{1/2}}  \Big).\end{align*}
et
\begin{align*} 
\sigma_2(x,y,M; a)&=  \frac{1}{ (2\pi i)^2} \int_{\alpha -iT_1}^{\alpha +iT_1}\int_{\kappa-iT_2}^{\kappa+iT_2}F_a(s_1,s_2;y) x^{s_1+s_2}(1-M^{-s_2})\frac{\d s_2\d s_1}{ s_1s_2}
\\ &\quad +O\Big(\frac{\Psi(x,y)}{ T_1^{1/3}}+\frac{\Psi(x,y)}{ T_2^{1/2}}\Big).\end{align*}
\end{lemma}

\begin{proof} Montrons en d\'etail la somme $\sigma_2(x,y,M; a)$ d\'efinie en \eqref{defsigma12}. Une interversion de sommation fournit
$$\sigma_2(x,y,M; a)=\sum_{n\in S(x,y)} \sum_{\substack{x/M<q\leq x\\ (q,na)=1}}\frac{1}{\phi(q)}.$$
En utilisant la majoration $n/\phi(n)\ll \log\log (3n)$, une  formule de Perron    \cite[corollaire II.2.3 et II.2.4]{T08} fournit lorsque $\kappa=1/\log x$
\begin{align*}
\sum_{\substack{x/M<q\leq x\\ (q,na)=1}}\frac{1}{\phi(q)} 
=&\frac{1}{ 2\pi i} \int_{\kappa-iT_2}^{\kappa+iT_2} \Phi_{2}(s_2;an) x^{s_2}(1-M^{-s_2})\frac{\d s_2}{ s_2}
\\ &+O\Big(\frac{(\log x)+(\log T_2)(\log\log x)}{ T_2}+
\frac{(\log\log x)M}{ x}\Big).\end{align*}
Nous prenons $ T_2\geq (\log x)^2$ de sorte que le premier terme d'erreur soit $\ll T_2^{-1/2}$. De plus, le terme principal est  $O(\log M)$.

Nous rappelons la d\'efinition \eqref{defFa} de $F_a(s_1,s_2;y).$
Une deuxi\`eme formule de Perron \cite[th\'eor\`eme II.2.3]{T08} fournit alors 
\begin{align*}
\sigma_2(x,y,M; a)&=\sum_{n\in S(x,y)} \sum_{\substack{x/M<q\leq x\\ (q,na)=1}}\frac{1}{\phi(q)} 
\cr&= \frac{1}{ (2\pi i)^2} \int_{\alpha -iT_1}^{\alpha +iT_1}\int_{\kappa-iT_2}^{\kappa+iT_2}F_a(s_1,s_2;y) x^{s_1+s_2}(1-M^{-s_2})\frac{\d s_2\d s_1}{ s_1s_2}
\\ &\quad +O\Big(\frac{\Psi(x,y)}{ T_2^{1/2}}+\frac{\Psi(x,y)\log\log x}{ x^{1/2}} + x^{\alpha }\log M\sum_{P(n)\leq y} \frac{1}{n^{\alpha  }(1+T_1|\log (x/n)|)}\Big).\end{align*}
La contribution des $n\notin[x-x/ T_1^{1/2},x+x/ T_1^{1/2}]$ dans la somme est 
$$ \ll x^{\alpha }\zeta(\alpha ;y)(\log M) /T_1^{1/2} .$$
En supposant $T_1\geq (\log x)^{12}$ et en utilisant le Lemme \ref{estcol} et \eqref{estsigma}, nous obtenons une contribution 
$  \ll \Psi(x,y)/T_1^{1/3}$. D'apr\`es la sous-convexit\'e de $\Psi(x,y)$  \cite[theorem 4]{H85},   la contribution des $n\in[x-x/ T_1^{1/2},x+x/ T_1^{1/2}]$ dans la somme est 
\begin{align*}&\ll(\log M)\big(\Psi(x+x/ T_1^{1/2},y)-\Psi(x-x/ T_1^{1/2},y)\big) \cr&\leq(\log M)\Psi(2x/ T_1^{1/2},y)\ll
   x^{\alpha }\zeta(\alpha,y)(\log M) T_1^{-\alpha /2} 
 \ll \Psi(x,y) /T_1^{1/3}  \end{align*}
dans le domaine $y\geq (\log x)^C$. L'avant-derni\`ere majoration d\'ecoule de la majoration de Rankin $\Psi(x,y)\leq x^\alpha\zeta(\alpha,y)$ et la derni\`ere majoration provient de \eqref{majcol} coupl\'e avec l'in\'egalit\'e $T_1\geq (\log x)^{12}$. 

 Cela ach\`eve la preuve de l'estimation de $\sigma_2(x,y,M; a)$.

Nous rappelons la d\'efinition \eqref{deftildes} de $\widetilde\sigma_1(x,y,M; a)$.
Gr\^ace au calcul \eqref{zetayell}, une formule de Perron fournit 
\begin{align*}\Psi_{a'r\ell_1}( x/k ,y) -\Psi_{a'r\ell_1}( rx/M ,y ) =&\frac{1}{ 2\pi i} \int_{\alpha -iT_1}^{\alpha +iT_1} g_{a'r_y\ell_1}(s_1)\zeta(s_1,y) \frac{x^{s_1}}{ k^{s_1}}  (1 -(kr/M)^{s_1})\frac{\d s_1}{ s_1}
\\ &+O\Big(\frac{\Psi(x,y)}{ T_1^{1/3}} \Big),\end{align*} 
o\`u $r_y$ est la partie $y$-friable de $r$ et $c>0$ une constante positive convenable.
Puis,  nous estimons la somme en $r$ gr\^ace une autre application de la formule de Perron. 
\`A partir de la formule
$$\qquad\qquad\frac{1}{ 2\pi i} \int_{\kappa-i\infty}^{\kappa+i\infty}     x^{s } \frac{\d s }{ s (s+w)}=\frac{1-x^{-w}}{w}{\bf 1}_{x\geq 1}\qquad\qquad (\kappa>0,\, x\in \RR_{>0}, \Re e(w)>0),$$
nous obtenons 
$$ \sum_{\substack{1\leq r<M/k\\ (r,a'n)=1}}\frac{\phi( \ell_1)g_{a'r_y\ell_1}(s_1) }{\phi(r\ell_1) }(1 -(kr/M)^{s_1})=
\frac{1}{ 2\pi i }  \int_{\kappa-i\infty}^{\kappa+i\infty}   \sum_{\substack{r=1}}^\infty\frac{\phi( \ell_1)g_{a'r_y\ell_1}(s_1) }{\phi(r\ell_1)r^{s_2} }\frac{ (M/k )^{s_2}s_1 \d s_2 }{ s_2(s_2+s_1)}.$$ 
Nous avons donc
\begin{align*} \widetilde\sigma_1  (x,y,M; a)    
&=\frac{1}{(2\pi i)^2}\int_{\alpha -iT_1}^{\alpha +iT_1} \!\!\!\int_{\kappa-i\infty}^{\kappa+i\infty}G_a(s_1,s_2;y)x^{s_1}M^{s_2} \frac{\d s_2\d s_1}{(s_2+s_1)s_2}\cr&\quad+O\Big(\tau(a) \frac{\Psi(x,y)}{ T_1^{1/3}}\ \Big),\end{align*} 
o\`u $G_a$ a \'et\'e d\'efini en \eqref{defGa}. En utilisant les Lemmes \ref{calculGa} et \ref{majpsi1}, nous avons dans les segments consid\'er\'es, lorsque $|\tau_2|\geq 1$,
$$G_a(s_1,s_2;y)\ll \tau(a)\zeta(\alpha,y).$$
Il est donc possible de tronquer la droite d'int\'egration en $s_2$ en un segment d\'efini par $|\tau_2|\leq T_2$ au prix d'un terme d'erreur 
$$\ll x^\alpha \zeta(\alpha,y)\frac{T_1(\log T_2)}{ T_2}\ll \frac{\Psi(x,y)}{T_2^{1/3}}.
$$
\end{proof}
o\`u nous avons utilis\'e \eqref{majzeta'/zeta}, puis \eqref{majcol}.
\goodbreak 

\section{Premi\`ere application du th\'eor\`eme des r\'esidus}

Nous introduisons la fonction $\widetilde G_a(s_1,s_2;y)$ construite \`a partir de la fonction  $G_a(s_1,s_2;y)$ introduite au Lemme \ref{calculGa} de la mani\`ere suivante 
$$
\widetilde G_a(s_1,s_2;y) : =G_a(s_1,s_2;y)\frac{\zeta(s_1+s_2+1,y)}{\zeta(s_1+s_2+1)}.$$ 
Le Lemme \ref{calculGa} coupl\'e avec \eqref{G1H1} et \eqref{calculpsi1} fournit
\begin{equation}\label{tildeG1H1}\begin{split}
\widetilde G_a(s_1,s_2;y) =\psi_{1}(s_1,s_2;a)\zeta(s_1,y)\frac{\zeta(s_2+1)\zeta(s_2+2)}{\zeta(s_1+s_2+1)}   H_1(s_1,s_2;y),\end{split}
\end{equation} 
o\`u une expression de $\psi_{1}(s_1,s_2;a) $ a \'et\'e donn\'ee en \eqref{calculpsi1}.
Une premi\`ere application du th\'eor\`eme des r\'esidus fournit l'estimation interm\'ediaire sui\-vante.
\begin{lemma} \label{lem2sigma}
Soit $A>0$. Il existe des constantes $C\geq  \tfrac52,$ $\delta>0$ telles  que,  lorsque  $(\log x)^C\leq y\leq x^{1/C}$, $1\leq |a|\leq M_1\leq y $,   $ (\log x)^{13}\leq T_1\leq T_2^{1/4}\leq \sqrt{x}$,   $(\log x)^{12}\leq T_2\leq \sqrt{y}$, $\kappa'=-\alpha +1/\log M_1$, nous avons 
\begin{align*}\sigma (x,y,M_1; a) =&\frac{1}{(2\pi i)^2}\int_{\alpha -iT_1}^{\alpha +iT_1} \!\!\!\int_{\kappa'-iT_2}^{\kappa'+iT_2}\widetilde G_a(s_1,s_2;y)x^{s_1}     M_1^{s_2}     \frac{\d s_2\d s_1}{ s_2( s_2+s_1)}
\cr&
  +O_{A}\Big(\tau_3(a)\Psi(x,y)\big\{ H(u)^{-\delta}(\log x)^{-A} +y^{-\delta}+T_1^{-1/3}+\tau_3(a) T_2^{-1/3}\big\}  \Big).\end{align*}
\end{lemma}

\begin{rem}
Dans le cas $M=1$ d\'evelopp\'e dans \cite{D13} pour $a=1$, le r\'esultat obtenu est sensiblement diff\'erent puisque 
$\sigma_3  (x,y,1; 1) =0$. Il est facile d'obtenir dans ce cas 
\begin{align*}\sigma (x,y,1; 1) =&Res_0(x,y;1)
 +O_{ A}\Big( \Psi(x,y)\big\{ H(u)^{-\delta}(\log x)^{-A} +y^{-\delta}+T_1^{-1/3}+  T_2^{-1/3}\big\}  \Big).\end{align*}
o\`u $Res_0(x,y;1)$ est calcul\'e  en \eqref{calculres0}. Cela compl\`ete la remarque 4 du Th\'eor\`eme \ref{th}.

\end{rem}
\begin{proof} 
 Nous \'evaluons la diff\'erence $\sigma_3 (x,y;M,a) $ des termes principaux des estimations de $\widetilde\sigma_1  (x,y;M,a) $ et $\sigma_2 (x,y;M,a)$ \'enonc\'ees dans le Lemme \ref{estsigma2} lorsque $M\leq x^{1/2}$. Nous avons 
\begin{align*} \sigma_3& (x,y,M; a)  \cr&=\frac{1}{(2\pi i)^2}\int_{\alpha -iT_1}^{\alpha +iT_1} \!\!\!\int_{\kappa-iT_2}^{\kappa+iT_2}x^{s_1}\Big( G_a(s_1,s_2;y)\frac{s_1 M^{s_2} }{ s_2+s_1 } -F_a(s_1,s_2;y)x^{s_2}(1-M^{-s_2})\Big) \frac{\d s_2\d s_1}{s_2s_1}
.\end{align*} 

Comme, en prenant la constante $C$ suffisamment grande, nous pouvons supposer $1-\alpha\leq\tfrac15$, nous d\'ecalons la droite d'int\'egration en la variable $s_2$ jusqu'\`a l'abscisse $-\varepsilon$ o\`u $\varepsilon$ est un param\`etre choisi suffisaemment petit. Dans la bande verticale d\'efinie par $-\tfrac16\leq \Re e s_2\leq \kappa$, d'apr\`es le Lemme \ref{majpsi1}, nous avons  
\begin{equation}\label{majFaGa}\bigg|\frac{ F_a(s_1,s_2;y) }{\zeta(s_2+1)\zeta(s_1,y)}\bigg|+\bigg|\frac{ G_a(s_1,s_2;y) }{\zeta(s_2+1)\zeta(s_1,y)}\bigg|\ll \tau(a).
\end{equation}
Le seul r\'esidu dans cette bande verticale correspond \`a $s_2=0$.
Nous \'ecrivons 
$$F_a(s_1,s_2;y)= f_a(s_1,s_2;y)/s_2,\qquad G_a(s_1,s_2;y)= g_a(s_1,s_2;y)/s_2.$$
D'apr\`es \eqref{calculG} et la relation $K_a( s_1,0 )=1$, nous avons 
\begin{align*}f_a(s_1,0;y) = g_a(s_1,0;y) .
\end{align*} 
Le r\'esidu en $s_2=0$ de la fonction intervenant dans l'\'ecriture de  $\sigma_3  (x,y,M; a)$ vaut
\begin{equation}\label{calculres0}Res_0(x,y;a)  =\frac{1}{2\pi i }\int_{\alpha -iT_1}^{\alpha +iT_1}   x^{s_1}\Big( -\frac{g_a(s_1,0;y)}{  s_1 } + \frac{\partial  g_a(s_1,0;y)}{\partial s_2} \Big) \frac{\d s_1 }{s_1 }
 .
\end{equation}
Ce terme  ne d\'epend pas de $M$. 
 
Par la majoration \eqref{majFaGa} et la majoration $|\zeta(s_1,y)|\leq \zeta(\alpha,y)$, nous obtenons  aussi 
$$
\frac{1}{(2\pi i)^2}\int_{\alpha -iT_1}^{\alpha +iT_1} \!\!\!\int_{-\varepsilon-iT_2}^{-\varepsilon+iT_2}F_a(s_1,s_2;y)x^{s_1+s_2}(1-M^{-s_2})  \frac{\d s_2\d s_1}{s_1s_2}\ll\tau(a)\frac{\Psi(x,y)M^{\varepsilon}T_2^{\varepsilon}}{ x^{\varepsilon}}(\log x)^2.$$
Pour justifier cette application, nous pouvons utiliser les bornes de la fonction $\zeta$ \'enonc\'ees au Lemme~\ref{zeta} et la majoration \eqref{majcol}.

Ainsi le th\'eor\`eme des r\'esidus fournit
 \begin{align*}  \sigma_3  (x,y,M; a)    =&Res_0(x,y;a)+O\Big(\tau(a)\frac{\Psi(x,y)M^{\varepsilon}T_2^{\varepsilon}}{ x^{\varepsilon/2}}+\tau(a)\frac{\Psi(x,y) }{ T_2^{1-\varepsilon}}\Big) \cr&+\frac{1}{(2\pi i)^2}\int_{\alpha -iT_1}^{\alpha +iT_1} \!\!\!\int_{-\varepsilon-iT_2}^{-\varepsilon+iT_2}G_a(s_1,s_2;y)x^{s_1}     M^{s_2}    \frac{\d s_2\d s_1}{ s_2( s_2+s_1)},\end{align*}
o\`u le second terme dans le terme d'erreur correspond aux contributions des segments horizontaux.
Le dernier terme de cette formule est $ \ll \tau(a)\Psi(x,y)(\log x)^2/M^{\varepsilon} $ et l'on peut n\'egliger ce terme lorsque $M=M_2$. 
En reportant ces r\'esultats dans \eqref{diffsigma} et en utilisant le Lemme ~\ref{lemsigma} et \eqref{qpetit}, nous obtenons l'estimation 
\begin{align*}\sigma (x,y,M_1; a) = &-\frac{1}{(2\pi i)^2}\int_{\alpha -iT_1}^{\alpha +iT_1} \!\!\!\int_{-\varepsilon-iT_2}^{-\varepsilon+iT_2}G_a(s_1,s_2;y)x^{s_1}     M_1^{s_2}     \frac{\d s_2\d s_1}{ s_2( s_2+s_1)}
\cr&
  +O\Big(\tau_3(a)\Psi(x,y)\big\{ H(u)^{-\delta}(\log x)^{-A} +y^{-\delta}+T_1^{-1/3}+T_2^{-1/3}+M_1^{\varepsilon}T_2^{\varepsilon}/x^{\varepsilon} \big\}  \Big).\end{align*}

D'apr\`es le Lemme \ref{majpsi1}, lorsque $\Re e s_1=\alpha $ et $\Re e s_2=-\varepsilon$, nous avons 
$$\widetilde G_a(s_1,s_2;y)-G_a(s_1,s_2;y)\ll \big| \zeta(s_1,y)\zeta(s_2+1) \big|y^{-\alpha  +\varepsilon}\tau(a)$$
de sorte que le remplacement de $G_a(s_1,s_2;y)$ par $\widetilde G_a(s_1,s_2;y)$ n'induit qu'un terme d'erreur 
$$\ll \tau (a)\Psi(x,y)T_2^\varepsilon  y^{-\alpha }(y/M_1)^{\varepsilon}(\log x)^2(\log y)
\ll \tau (a)\Psi(x,y) y^{-\delta}  .$$

Il reste \`a observer que nous pouvons d\'ecaler l'abscisse d'int\'egration en $s_2$ jusqu'\`a  l'ab\-scisse $\kappa'=-\alpha +1/\log M$ en appliquant le th\'eor\`eme des r\'esidus \`a un rectangle. 
C'est possible puisque  $$\mu_\zeta(1+\sigma_2)+\mu_\zeta(2+\sigma_2)\leq 1/2\qquad \qquad (\kappa'\leq \sigma_2\leq \kappa),$$ 
o\`u nous renvoyons au Lemme \ref{zeta} pour la d\'efinition de $\mu_\zeta$.
D'apr\`es \eqref{maj2psi1}, la contribution des c\^ot\'es horizontaux est 
$$\ll \Psi(x,y)\frac{\tau_3(a)^2}{T_2^{1/2}}(\log x)^2\ll \Psi(x,y)\frac{\tau_3(a)^2}{T_2^{1/3}}, $$   pourvu que $T_2\geq (\log x)^{12}.$
Dans le rectangle consid\'er\'e, il n'y a pas de p\^ole puisque $0<\alpha <1$. Cela ach\`eve la d\'emonstration du Lemme \ref{lem2sigma}.
\end{proof} 

\section{Seconde application du th\'eor\`eme des r\'esidus}

Nous montrons maintenant le lemme suivant. Nous notons 
\begin{equation}\label{defI}I(x,y;M):= \frac{1}{2\pi i}\int_{\alpha -iL_{\varepsilon}(M)}^{\alpha +iL_{\varepsilon}(M)}   
 \frac{ \zeta(s_1,y)}{\zeta(s_1)(s_1-1)} x^{s_1-1}   \d s_1.\end{equation}

\begin{lemma} \label{lem3sigma}
Soient $\varepsilon\in\, ]0, \tfrac12[$ et $A>0$. Il existe des constantes $C\geq  \tfrac52$, $\delta>0$ telles  que,    lorsque $(x,y)$ satisfait  $(H_\varepsilon)$ et $  y\leq x^{1/C}$,  $1\leq |a|\leq M^{1/2-\varepsilon}$, $ M$ satisfait \`a \eqref{inegM},  
  nous avons 
\begin{equation}\label{elemme7.1}\sigma (x,y,M; a) =  -\frac{\phi(a)x}{2Ma}I(x/a,y;M) + O_{\varepsilon,A}\bigg( \frac{\tau_3(a)^2\Psi(x,y)}{ M L_{\varepsilon }(M)}  \bigg).\end{equation}
\end{lemma}

\begin{proof}  
 D'apr\`es \eqref{maj2psi1}, la fonction introduite $\widetilde  G_a$ en \eqref{tildeG1H1} v\'erifie  dans le domaine $\Re e(s_1)=\alpha$ et $\Re e(s_2)\geq-\alpha$
$$\widetilde G_a(s_1,s_2;y)\ll  \tau_3(a)^2{ |\zeta( s_1 ;y)|} \frac{ |\zeta(s_2+1)\zeta(s_2+2)|}{|\zeta(s_2+s_1+1 )|} $$ 
et, avec le Lemme \ref{zeta} et \eqref{majzeta'/zeta},
\begin{equation}\label{majGimportante}\bigg|    \frac{\widetilde G_a(s_1,s_2;y)x^{s_1}     M^{s_2} }{ s_2( s_2+s_1)}\bigg|\ll  \tau_3(a)^2 \frac{  |\zeta( s_1 ;y) | x^{\alpha }M^{ \sigma_2}}{(|\tau_1+\tau_2|+1)|\tau_2|^{1/2-\varepsilon}}\log (|\tau_1|+|\tau_2|+2). 
\end{equation}

Nous appliquons le Lemme \ref{lem2sigma} avec $T_1^{4}=T_2=(\log x)^{44}M^{16}$ de sorte que, sous la condition~\eqref{inegM}, nous ayons
\begin{equation}\label{estsigmae71}\begin{split}\sigma (x,y,M_1; a) =&\frac{1}{(2\pi i)^2}\int_{\alpha -iT_1}^{\alpha +iT_1} \!\!\!\int_{\kappa'-iT_2}^{\kappa'+iT_2}\widetilde G_a(s_1,s_2;y)x^{s_1}     M_1^{s_2}     \frac{\d s_2\d s_1}{ s_2( s_2+s_1)} 
  \cr&+O\Bigg(\frac{\tau_3(a)^2\Psi(x,y) }{M L_{\varepsilon/2}(M)} \Bigg).\end{split}\end{equation}

Pour  restreindre le segment  d'int\'egration en $s_2$ \`a $|\tau_2|\leq L_{\varepsilon}(M)$ et le   segment  d'int\'egration en $s_1$ \`a $|\tau_1|\leq  L_{\varepsilon}(M)$, nous distinguons deux cas.

Lorsque $u\geq (\log_2y)^2$ et sous la condition $T_1\leq H(u)^{ 4\delta} (\log x)^{ 4A+11} $ issue de \eqref{inegM}, le Lemme \ref{estcol2} fournit  
\begin{equation}\label{e7.5}\int_{\alpha -iT_1}^{\alpha +iT_1}\big|\zeta(s_1,y) x^{s_1} \big| |\d s_1| \ll 
\Psi(x,y) 
. \end{equation}
La restriction au  segment  d'int\'egration en $s_2$ \`a $|\tau_2|\leq L_{\varepsilon}(M)$  induit un terme d'erreur 
$$\ll \frac{\tau_3(a)^2}{M^{\alpha } L_{\varepsilon}(M)^{1/3}}\int_{\alpha -iT_1}^{\alpha +iT_1}\big|\zeta(s_1,y)  x^{s_1} \big| |\d s_1|\ll \frac{\tau_3(a)^2\Psi(x,y)}{M^{\alpha } L_{2\varepsilon/3}(M)}\ll  \frac{\tau_3(a)^2\Psi(x,y)}{M L_{\varepsilon/2}(M)},$$
o\`u la derni\`ere majoration a \'et\'e obtenue gr\^ace \`a $(x,y)\in (H_\varepsilon)$ et $M\leq y$.
Puis, la restriction au  segment  d'int\'egration en $s_1$ \`a $|\tau_1|\leq  L_{\varepsilon}(M)$  induit un terme d'erreur 
$$\ll \frac{\tau_3(a)^2}{M^{\alpha } L_{\varepsilon}(M)^{1/3}}\int_{\alpha -iT_1}^{\alpha +iT_1}\big|\zeta(s_1,y)  x^{s_1} \big| |\d s_1|\ll  \frac{\tau_3(a)^2\Psi(x,y)}{M L_{\varepsilon/2}(M)}.$$ 

Le cas $ u\leq (\log_2y)^2$ est plus d\'elicat car nous ne pouvons plus utiliser la majoration \eqref{majGimportante} et la formule \eqref{e7.5} n'est plus disponible.
Dans la formule
$$\widetilde G_a(s_1,s_2;y)=\psi_1(s_1,s_2;a)\zeta(s_1,y)\frac{\zeta(s_2+1)\zeta(s_2+2)}{\zeta(s_1+s_2+1)}   H_1(s_1,s_2;y),$$
 nous \'evaluons le facteur $\zeta(s_1,y)$ gr\^ace au Lemme \ref{zetasy} et \`a l'estimation \eqref{majtaugrd}
$$ (s_1-1)(\log y)\widehat\rho((s_1-1)\log y)=1+O\Big(\frac{1+u\xi(u)}{|s-1|\log y}\Big)   $$ 
valable lorsque $ |\tau_1|\geq (1+u\xi(u))/\log y $ donc lorsque  $ |\tau_1|\geq  L_{\varepsilon}(M)$.
Il vient 
\begin{equation}\label{estzetay}\zeta(s_1,y)=\zeta(s_1 )\Big\{1+O\Big(\frac{1+u\xi(u)}{(s_1-1)\log y}+\frac{1}{L_{\varepsilon}(y)}\Big) 
\Big\}.
\end{equation}
Ainsi la contribution au terme principal de \eqref{estsigmae71} des $ |\tau_1|\geq  L_{\varepsilon}(M)$ vaut
\begin{equation}\label{est7.4}\begin{split}  \frac{1}{(2\pi i)^2}& \int_{\kappa'-iT_2}^{\kappa'+iT_2}\!\!\!\int_{L_{\varepsilon}(M)< |\tau_1|\leq T_1} \!\!\!\!\!\!\!\!\!\!\!\!\!\!\!\!\!\!\!\!
\psi_1(s_1,s_2;a)\zeta(s_1 )\frac{\zeta(s_2+1)\zeta(s_2+2)}{\zeta(s_1+s_2+1)}   H_1(s_1,s_2;y)
    \frac{x^{s_1}     M^{s_2} \d s_2\d s_1}{ s_2( s_2+s_1)} 
  \cr&+O\Bigg(\frac{\tau_3(a)^2x^{\alpha} (1+u\xi(u))}{M^\alpha L_{\varepsilon }(M)^{1/3}\log y}+\frac{\tau_3(a)^2 (\log T_1)^2x^{\alpha}}{M^\alpha L_{\varepsilon }(y) } \Bigg),\end{split}\end{equation}
o\`u dans l'int\'egrale $s_1$ d\'esigne $\alpha+i\tau_1$.
Dans le domaine $(H_\varepsilon)$, d'apr\`es \eqref{estHild}, \eqref{estalpha} et \eqref{estrhou}, nous avons 
\begin{equation}\label{majxalpha}x^{\alpha} (1+u\xi(u))\ll x\e^{-u\xi(u)} (1+u\xi(u))\ll \Psi(x,y)\e^{-u/2} 
,\end{equation}
ce qui permet de majorer le terme d'erreur de mani\`ere satisfaisante.
La deuxi\`eme \'etape consiste  \`a utiliser l'approximation suivante
\begin{equation}\label{estzetas}\qquad\qquad
\zeta(s_1)=\sum_{1\leq n< |\tau_1|} \frac{1}{n^{s_1}}+O\Big(\frac{1}{ (|\tau_1|+1)^{\sigma_1}}\Big)\qquad\qquad (|\tau_1|\geq1).
\end{equation}
Cette m\'ethode a d\'ej\`a \'et\'e utilis\'ee par Saias \cite{S89} dans la preuve de la Proposition 2 (voir aussi la d\'emonstration du th\'eor\`eme III.5.17 de \cite{T08}).
En reportant celle-ci dans \eqref{est7.4}, nous montrons de la m\^eme mani\`ere que pr\'ec\'edemment que la contribution du terme d'erreur est englob\'ee dans le terme d'erreur de \eqref{elemme7.1}.

La contribution dans \eqref{est7.4} du terme principal de \eqref{estzetas} vaut alors 
\begin{equation}\label{sumsurn}\sum_{1\leq n\leq T_1} \int_{\kappa'-iT_2}^{\kappa'+iT_2} \zeta(s_2+1)\zeta(s_2+2) R_n(s_2,x/n;L_{\varepsilon}(M),T_1,a)
    \frac{     M^{s_2} \d s_2 }{ s_2 }.\end{equation}
avec
$$R_n(s_2,z;L ,T_1,a):= \frac{1}{ (2\pi i )^2}  \int_{\max\{ L ,n\}<|\tau_1|\leq T_1} \!\!\!\!\!\!\!\!\!
\psi_1(s_1,s_2;a) \frac{ H_1(s_1,s_2;y)z^{s_1}      \d s_1 }{\zeta(s_1+s_2+1)( s_2+s_1)}.$$
Dans chacun de ces int\'egrales d\'ependant de $n$, nous d\'ecalons l'abscisse d'int\'egration jusqu'\`a $\alpha-(\log T_1)^{-2/3-\varepsilon}-1/\log x$ afin que $s_1+s_2+1$ reste dans la r\'egion sans z\'ero de z\^eta \eqref{KV}.
Nous avons
\begin{equation}\label{majRn}R_n(s_2,z;L ,T_1,a)\ll_{\varepsilon} \tau_3(a)^2z^{\alpha}
\Big\{\Big(\frac{z}{a^2 }\Big)^{-(\log T_1)^{-2/3-\varepsilon}}(\log T_1)^2
+  \frac{T_n(\tau_2)^{-1+\varepsilon}}{\log z}
\Big\},
\end{equation}
avec 
$$T_n(\tau_2):=1+\min\big\{|\max\{ L ,n\}\pm \tau_2|, | T_1\pm \tau_2|\big\}.$$
Le second terme correspond \`a la contribution des segments horizontaux o\`u nous avons utilis\'e la majoration 
$$\int_{\alpha-(\log T_1)^{-2/3-\varepsilon}-1/\log x}^{\alpha} z^{\sigma_1}\d \sigma_1\ll \frac{z^\alpha}{\log z}.$$
La contribution du premier terme du majorant de $R_n$ dans la somme \eqref{sumsurn}  est
$$\ll \tau_3(a)^2 \frac{x^{\alpha-\tfrac12(\log T_1)^{-2/3-\varepsilon}}}{M^{ \alpha} } \frac{T_1^{1-\alpha}T_2^{1/2+\varepsilon}}{1-\alpha}\ll  \tau_3(a)^2\frac{\Psi(x,y)}{M^2}.
$$
compte-tenu de la taille relative de $T_1,$ $T_2$, $x$ et $M$.
D'apr\`es la majoration
$$\int_{-\infty}^\infty(1+|T\pm \tau_2|)^{-1+\varepsilon}(1+| \tau_2|)^{-1/2+\varepsilon}\d \tau_2\ll T^{-1/2+2\varepsilon}\qquad (T\geq 1),$$
la contribution du second terme du majorant de $R_n$ dans la somme \eqref{sumsurn}  est
$$\ll \tau_3(a)^2 \frac{x^{\alpha}}{M^{ \alpha}\log x} \sum_{1\leq n\leq T_1}\frac{1}{ n^\alpha \max\{ L_{\varepsilon}(M),n\}^{1/2-2\varepsilon}}\ll
 \frac{\tau_3(a)^2\Psi(x,y)}{M L_{\varepsilon/2}(M)}
$$

Montrons maintenant que nous pouvons n\'egliger la contribution dans \eqref{estsigmae71} du segment  d'int\'egration en $s_2$ \`a $|\tau_2|\leq L_{\varepsilon}(M)$. 
Les d\'etails sont tr\`es proches du cas pr\'ec\'edent. La contribution des $|\tau_1|\leq (1+u\xi(u))/\log y$ et $L_{\varepsilon}(M)\leq |\tau_2|\leq T_2$ est n\'egligeable gr\^ace \`a la seconde estimation du Lemme \ref{estcol2}.  
Lorsque $(1+u\xi(u))/\log y<|\tau_1|\leq L_{\varepsilon}(M)$, nous pouvons utiliser successivement la formule \eqref{estzetay} puis \eqref{estzetas}.
La contribution des termes d'erreur de chacune de ces estimations est facilement n\'egligeable gr\^ace \`a \eqref{majxalpha}.  Il reste donc \`a majorer
\begin{equation}\label{??}\sum_{1\leq n\leq L_{\varepsilon}(M)} \int_{L_{\varepsilon}(M)\leq |\tau_2|\leq T_2}\!\!\!\!\!\!\!\!\!\!\!\!\!\!\!\!\!\! \zeta(s_2+1)\zeta(s_2+2) R_n(s_2,x/n; (\log (u+1))/\log y,L_{\varepsilon}(M),a)
    \frac{     M^{s_2} \d s_2 }{ s_2 }.\end{equation}
Pour cela,  nous utilisons la majoration \eqref{majRn} de $R_n$ pour obtenir une contribution major\'ee par $$\ll  \frac{\tau_3(a)^2\Psi(x,y)}{M L_{\varepsilon/2}(M)}.$$ 
 
Nous obtenons donc
\begin{equation}\label{estsigmaintegrande}\begin{split}\sigma (x,y,M;& a) = -\frac{1}{(2\pi i)^2}\int_{\alpha -iL_{\varepsilon}(M)}^{\alpha +iL_{\varepsilon}(M)} \!\!\!\int_{\kappa'-iL_{\varepsilon}(M)}^{\kappa'+iL_{\varepsilon}(M)}\widetilde G_a(s_1,s_2;y)x^{s_1}     M^{s_2}     \frac{\d s_2\d s_1}{ s_2( s_2+s_1)}
\cr&\qquad+O \bigg( \frac{\tau_3(a)^2\Psi(x,y)}{ M L_{\varepsilon/2 }(M)}  \bigg).\end{split}\end{equation}

Dans une deuxi\`eme application du th\'eor\`eme des r\'esidus, nous d\'ecalons la droite d'int\'e\-gration en $s_2$ dans la double int\'egrale approchant $\sigma (x,y,M; a) $   jusqu'\`a 
\begin{equation} \Re e( s_2)=\kappa'':= 
-\alpha -\frac{c}{(\log  M)^{2/5+\varepsilon/3}} <-1
\label{inegdomaine}\end{equation}   avec $|\tau_2|\leq L_{\varepsilon}(M) $ et $c$ suffisamment petite. 

 \'Etant-donn\'e nos connaissances sur la r\'egion  sans z\'ero de la fonction z\^eta de Riemann~\eqref{KV} et l'expression \eqref{tildeG1H1} d\'efinissant $\widetilde G_a(s_1,s_2;y)$, le seul p\^ole contenu \`a l'int\'erieur du rectangle consid\'er\'e est celui en $s_2=-1$.  
 Nous notons $Res_1$   le r\'esidu correspondant \`a~$s_2=-1$.

Nous avons
$$Res_1= \frac{1}{ 2\pi i }\int_{\alpha -iL_{\varepsilon}(M)}^{\alpha +iL_{\varepsilon}(M)}(\widetilde G_a(s_1,s_2;y)\zeta(s_2+2)^{-1}\big)(s_2=-1)\frac{x^{s_1}}{M}       \frac{\d s_1 }{ s_1-1}.$$
avec, d'apr\`es \eqref{tildeG1H1} et \eqref{calculG}, et la formule $H_1(s_1,-1)=1$,
\begin{align*}(G_a(s_1,s_2;y)\zeta(s_2+2)^{-1}\big)(s_2=-1)&=\frac{\phi(a)}{a^{s_1}} \zeta(0)
 \frac{ \zeta(s_1,y)}{\zeta(s_1)} =-\frac{\phi(a)}{2|a|^{s_1}}  
 \frac{ \zeta(s_1,y)}{\zeta(s_1)} 
,
\end{align*} 
o\`u nous avons utilis\'e la formule $\zeta(0)=-1/2$. La premi\`ere \'egalit\'e se justifie par la formule 
$$\Big(\frac{\psi_2(s_1,s_2;a)}{g_a(s_2+1)}\Big)(s_2=-1)=\frac{g_a(1)}{g_a(s_1)}
$$
issue du Lemme \ref{calculFa} et par la formule 
$$\big( {K_a(s_1,s_2 ) g_a(s_2+1)}\big)(s_2=-1)=\frac{g_a(s_1)}{g_a(s_1-1)}\prod_{p^\nu\parallel a}\Big(\frac{1-p^{-s_1+1}}{p^{\nu(s_1-1)}}\Big),$$
issue de \eqref{estKa}.
Ainsi $$Res_1=-\frac{\phi(|a|)x}{2M|a|}I(x/|a|,y;M) .$$

Nous pouvons choisir $c$ et $\varepsilon$ suffisamment petit de sorte que 
$$\mu_\zeta(1+\sigma_2)+\mu_\zeta(2+\sigma_2)\leq  \tfrac23-\varepsilon\qquad \qquad (  \Re e(s_2)=\sigma_2\geq \kappa''),$$ 
 D'apr\`es le  Lemme \ref{zeta}, \eqref{maj2psi1} et \eqref{majzeta'/zeta},  nous avons dans la r\'egion $\Re e(s_1)=\alpha$, $\Re e(s_2)\geq \kappa''$,  
$$\bigg|    \frac{\widetilde G_a(s_1,s_2;y)x^{s_1}     M^{s_2} }{ s_2( s_2+s_1)}\bigg|\ll  \tau_3(a)^2 \frac{  |\zeta( s_1 ;y) | x^{\alpha }M^{\sigma_2}|a|^{-2(\sigma_2+\alpha)}}{(|\tau_1+\tau_2|+1)(|\tau_2|+1)^{1/3}}\log (|\tau_1+\tau_2|+2). $$  

Majorons la contribution sur les deux c\^ot\'es horizontaux du rectangle et du c\^ot\'e vertical d'abscisse $\sigma_2=\kappa''$.

Lorsque $u\geq (\log_2y)^2$, nous pouvons utiliser
\eqref{e7.5} puisque $L_{\varepsilon}(M)\leq T_1$. 
Nous majorons la contribution des segments horizontaux  par
$$\ll \frac{\tau_3(a)^2\Psi (x,y)}{M^\alpha L_{\varepsilon}(M)^{1/4}}  \int_{\kappa''}^{\kappa'}\Big(\frac{M}{a^2}\Big)^{ \sigma_2+\alpha}\d \sigma_2
\ll_{\varepsilon}  \frac{\tau_3(a)^2\Psi (x,y)}{ML_{\varepsilon}(M)^{1/8}}  
.$$   Pour borner la contribution du segment vertical correspondant \`a l'abscisse $\Re e(s_2)=\kappa''$, nous s\'eparons ce domaine d'int\'e\-gration en la r\'egion $\{ |\tau_1+\tau_2| \leq |\tau_2|/2, |\tau_1|,|\tau_2| \leq L_{\varepsilon}(M) \}$ et son compl\'ement. Dans ces deux r\'egions, nous obtenons une contribution
$$\ll  \frac{\tau_3(a)^2\Psi (x,y)}{M^{\alpha}} (M/a^2)^{\kappa''+\alpha}  
\ll_{\varepsilon} \frac{\tau_3(a)^2\Psi (x,y)}{ML_{\varepsilon/2}(M) } . 
 $$ Dans ces deux derni\`eres \'etapes, la condition $|a|\leq M^{1/2-\varepsilon}$ a jou\'e une r\^ole crucial.

Lorsque $u\leq (\log_2y)^2$, les d\'etails sont \`a nouveau plus d\'elicats.
La contribution des $|\tau_1|\leq (1+u\xi(u))/\log y$  est n\'egligeable gr\^ace \`a la seconde estimation du Lemme \ref{estcol2} qui remplace la majoration \eqref{e7.5}.  
Lorsque $(1+u\xi(u))/\log y\leq |\tau_1|\leq L_{\varepsilon}(M)$, nous pouvons utiliser  \eqref{estzetay}.  

Nous pouvons alors majorer la contribution du segment horizontal d'ordonn\'ee $ L_{\varepsilon}(M)$ par 
\begin{align*}&\ll \frac{\tau_3(a)^2x^\alpha}{M^\alpha L_{\varepsilon}(M)^{1/4}}  \int_{\kappa''}^{\kappa'}\Big(\frac{M}{a^2}\Big)^{ \sigma_2+\alpha}\d \sigma_2   
\int_{|\tau_1|\leq  L_{\varepsilon}(M)}\frac{|\zeta(s_1)|}{
|L_{\varepsilon}(M)+\tau_1|+1}\d \tau_1
\cr& \ll \frac{\tau_3(a)^2x^\alpha}{M^\alpha L_{\varepsilon}(M)^{1/5}}\ll
\frac{\tau_3(a)^2\Psi(x,y)}{M  L_{\varepsilon/2}(M) }.
\end{align*}  
Il en est de m\^eme pour le segment horizontal d'ordonn\'ee $-L_{\varepsilon}(M)$.
La contribution du segment vertical est
\begin{align*}&\ll\tau_3(a)^2\frac{x^{\alpha }M^{\kappa''}}{|a|^{2(\sigma_2+\alpha)}}\int_{|\tau_2|\leq  L_{\varepsilon}(M)}\int_{|\tau_1|\leq  L_{\varepsilon}(M)} \frac{  |\zeta( s_1 ) | \log (|\tau_1+\tau_2|+2)}{(|\tau_1+\tau_2|+1)(|\tau_2|+1)^{1/3}}\d \tau_1\d\tau_2\cr&\ll
\tau_3(a)^2\frac{x^{\alpha }M^{\kappa''}}{|a|^{2(\sigma_2+\alpha)}}L_\varepsilon(M)\ll\frac{\tau_3(a)^2\Psi (x,y)}{ML_{\varepsilon/2}(M) },\end{align*}
o\`u nous avons utilis\'e \eqref{majxalpha} pour majorer $x^{\alpha }$.

Nous obtenons ainsi dans tout le domaine $(H_\varepsilon)$ l'estimation
\begin{align*}\sigma (x,y,M; a) =& -\frac{\phi(|a|)x}{2M|a|}I(x/|a|,y;M)    
  +O_{\varepsilon,A}\Bigg(\frac{\tau_3(a)^2\Psi (x,y)}{ML_{\varepsilon/2}(M) }   \Bigg).\end{align*}   
\end{proof}
  
La quantit\'e $I(x,y;M)$ introduite en \eqref{defI}
est une approximation naturelle de la somme 
$$\Phi_{\mu}(x,y):=\sum_{\substack{n\leq x\\ P^-(n)>y}}\frac{\mu(n)}{n},$$
o\`u $P^{-}(n)$ d\'esigne le petit facteur premier divisant $n$ avec la convention habituelle $P^-(1)=+\infty$.
En effet,
la formule de Perron fournit lorsque $x\notin \NN$ 
\begin{align*}\Phi_{\mu}(x,y)=\frac{1}{2\pi i}\int_{\kappa-i\infty}^{\kappa+i\infty}   
 \frac{ \zeta(s_1,y)}{\zeta(s_1)(s_1-1)} x^{s_1-1}       \d s_1\qquad (\kappa>1).\end{align*}
 Une formule de Perron tronqu\'ee fournit classiquement
$$ \Phi_{\mu}(x,y) =I(x,y;y)+O\Big(\frac{1}{L_\varepsilon(y)}\Big).$$
Dans \cite[\'equation (1.5)]{LT14}, il est montr\'e l'estimation
\begin{equation}\label{estPhimu}\begin{split}\Phi_{\mu}(x,y)& 
 =\Big\{ 1+O\Big(\frac{\log (u+1)}{\log y}\Big)\Big\} \rho(u)+O\Big(\frac{1}{L_\varepsilon(y)}\Big)\end{split}\qquad ((x,y)\in G_\varepsilon).\end{equation} 
Ici, $(G_\varepsilon)$ d\'esigne le domaine en $(x,y)$ d\'efini par 
$$x\geq 2,\qquad \exp\{ (\log x)^{2/5+\varepsilon}\big\}\leq y\leq x.$$

Nous reprenons la m\'ethode de \cite{LT14} en pr\'ecisant ce r\'esultat. 

\begin{lemma}\label{estI} Soit $\varepsilon\in\, ]0,1/2[$.
Lorsque $(x,y)$ satisfait  $(H_\varepsilon)$ et $1\leq M\leq y$, nous avons
$$I(x,y;M)=\rho(u)\Big\{1+O\Big( 
\frac{1}{H(u)^{ \delta}L_{\varepsilon/2}(M) }+\frac{1}{ L_{\varepsilon/2}(y) }
\Big)\Big\},
$$
o\`u $\delta$ est une constante positive.
Ainsi, lorsque $(x,y)$ est dans $(G_\varepsilon)$, nous avons  $$\Phi_{\mu}(x,y;y)=\rho(u) +O\Big(\frac{1}{L_{\varepsilon/2}(y)}\Big).$$
\end{lemma}

\begin{rem} 1-- L'extension de l'approximation de $I(x,y;M)$ \`a un domaine plus vaste que le domaine d'Hildebrand $(H_\varepsilon)$   demanderait une   r\'egion sans z\'ero de la fonction $\zeta$ de Riemann  plus grande. Nous renvoyons pour cela \`a l'article de Hildebrand \cite{H84} expliquant le lien entre le domaine de l'approximation de $\Psi(x,y)$ par $x\rho(u)$ et la r\'egion sans z\'ero de la fonction $\zeta$ de Riemann.

\par 2-- La restriction \`a $(G_\varepsilon)$ vient de la premi\`ere estimation de \eqref{estPhimu} \'etablie dans \cite{LT14}.
\end{rem}
\begin{proof} Nous montrons tout d'abord que, gr\^ace \`a une application du th\'eor\`eme des r\'esidus et \`a la formule \eqref{estalpha}, nous pouvons d\'ecaler l'abscisse d'int\'egration de $\alpha$ \`a $\alpha_0=1-\frac{\xi (u)}{ \log 
y}$. 
Le terme d'erreur induit correspond \`a la contribution des segments verticaux d'ordonn\'ee $\pm L_\varepsilon(M)$ et est major\'e par
$$\ll x^{ \alpha_0-1}\zeta(\alpha,y)\frac{|\alpha-\alpha_0|}{L_\varepsilon(M)^{1/2}}\ll \frac{\rho(u)}{L_\varepsilon(M)^{1/2}},$$
o\`u nous avons utilis\'e \eqref{estalpha} puis \eqref{majcol}. En effet, la majoration $|\alpha-\alpha_0|\ll 1/\log x$ implique $x^{\alpha}\asymp x^{\alpha_0}$ et $\zeta(\alpha,y)\asymp \zeta(\alpha_0,y)$.
Nous approchons le rapport $$ \frac{\zeta(s_1,y)}{\zeta(s_1)(s_1-1)}  $$ apparaissant dans la d\'efinition \eqref{defI} de $I(x,y;M)$ lorsque $s_1=\alpha_0 +i\tau_1$ et $ |\tau|\leq L_{\varepsilon}(M)\leq L_{\varepsilon}(y)$   gr\^ace au Lemme~\ref{zetasy}.    L'appartenance \`a $(H_\varepsilon)$ permet de se placer dans le domaine d'application de cette approximation. 
 
Nous avons ainsi
\begin{equation}\label{estrapportzeta}  \frac{\zeta(s_1,y)}{\zeta(s_1)(s_1-1)} =\widehat \rho((s_1-1)\log y)(\log y)
\Big\{ 1+O\Big(\frac{1}{L_{\varepsilon}(y)} \Big)\Big\} . \end{equation} 

Les estimations \eqref{majrhou}, \eqref{majtaugrd}, \eqref{estrhou} du Lemme \ref{lemmarho} fournissent
$$\int_{-L_{\varepsilon}(M)}^{ L_{\varepsilon}(M)}   
 |\widehat\rho(-\xi(u)+i\tau ) |\e^{-u\xi(u)}       \d \tau \ll \rho(u)+\rho(u)H(u)^{-1}\log\Big( \frac{L_{\varepsilon}(M)+\xi(u)}{1+u\xi(u)} \Big).$$
La contribution du terme d'erreur de \eqref{estrapportzeta} est donc
$$\ll \frac{  \rho(u)}{L_{\varepsilon/2}(y)} .$$ 
D'apr\`es \eqref{majtaugrd}, nous avons lorsque $T>1+u\xi(u)$
\begin{equation}\label{majint1}\int_{T}^\infty \widehat\rho(-\xi(u)+i\tau ) \e^{iu\tau}       \d \tau\ll \frac{1+u\xi(u)}{T}
\end{equation}
de sorte que pour tout $T\geq 1$ et $u\geq 1$ gr\^ace \`a \eqref{estrhou} et \eqref{estIxi}
\begin{equation}\label{majint2}\int_{T}^\infty \widehat\rho(-\xi(u)+i\tau ) \e^{iu\tau}       \d \tau\ll \frac{\rho(u)\e^{u\xi(u)}H(u)^{-\delta}}{T}.
\end{equation}
Notons que lorsque $1\leq T\leq 1+u\xi(u)$ cette majoration  d\'ecoule aussi  de \eqref{majrhou}.

Le terme principal attendu  dans l'approximation de 
$ I(x,y;M)$ est donc
$$ 
 \frac{1}{ 2\pi i }\int_{-\xi(u)-iL_\varepsilon  (M)\log y }^{-\xi(u)+iL_\varepsilon  (M)\log y  } \widehat\rho(s) \e^{u s }         \d s 
   .  $$ 
Nous \'etendons le segment d'int\'egration en la droite $\Re e (s)= -\xi(u)
$. 
L'erreur peut  donc \^etre major\'ee   par
$O\big(\rho(u) H(u)^{-\delta}L_{\varepsilon/2}(M)^{-1}\big).$ 
En prenant $T=L_\varepsilon  (M)\log y$ dans \eqref{majint1} et~\eqref{majint2}, nous obtenons donc 
$$ 
 \frac{1}{ 2\pi i }\int_{-\xi(u)-iL_\varepsilon  (M)\log y }^{-\xi(u)+iL_\varepsilon  (M)\log y  } \widehat\rho(s) \e^{u s }         \d s 
 =\rho(u)\big\{ 1+O\big( H(u)^{-\delta}L_{\varepsilon/2}(M)^{-1}\big)\big\}  .  $$  
Ainsi
$$I(x,y;M)=\rho(u)\Big\{1+O\Big( \frac{1}{H(u)^{ \delta}L_{\varepsilon/2}(M) }+\frac{1}{ L_{\varepsilon/2}(y) }\Big)\Big\}.$$
En reprenant la d\'emonstration de \cite{LT14}, cela pr\'ecise l'estimation \eqref{estPhimu} de la m\^eme mani\`ere
\begin{equation}\begin{split}\Phi_{\mu}(x,y) &=  \rho(u)+O\Big(\frac{1}{L_\varepsilon(y)}\Big) \end{split}\qquad ((x,y)\in G_\varepsilon).\end{equation}
\end{proof}

\noindent{\it D\'emonstration du Th\'eor\`eme \ref{th}.} En reportant le r\'esultat du Lemme \ref{estI} dans le Lem\-me~\ref{lem3sigma}, nous obtenons dans le domaine $(H_\varepsilon)$ l'estimation du Th\'eor\`eme \ref{th}.

\bigskip
\noindent{\bf{\textit{Remerciements.}}} Le premier auteur est soutenu par une bourse IUF junior, et le deuxi\`eme  par une bourse postdoctorale de la Fondation Sciences Math\'ematiques de Paris. Les auteurs tiennent \`a remercier chaleureusement l'IMJ-PRG pour les conditions de travail exceptionnelles. Ce travail a b\'en\'efici\'e de discussions avec Sary Drappeau et de remarques de G\'erald Tenenbaum, que nous remercions.


\begin{thebibliography}{cc}
\bibitem{BT05} {\sc R. de la Bret\`eche \& G. Tenenbaum,}  
 Propri\'et\'es statistiques des entiers friables,  
 {\it  Ramanujan Journal},  {\bf 9}  (2005),  n$^{\circ}$\thinspace 1-2, 139--202. 


\bibitem{D13} {\sc S. Drappeau}, {Th\'eor\`eme de Fouvry--Iwaniec pour les entiers friables}, {\it Compos. Math.}, \`a paraitre (2014).
 
\bibitem{F12} 
 {\sc D. Fiorilli},   Residue classes containing an unexpected number of primes, {\it  Duke Math. J.} {\bf 161} (2012), no. 15, 2923--2943.

\bibitem{F13} 
 {\sc D. Fiorilli}, The influence of the first term of an arithmetic progression, {\it  Proc. Lond. Math. Soc.} (3) {\bf 106} (2013), no. 4, 819--858.

\bibitem{FT96} {\sc \'E. Fouvry \& G. Tenenbaum}, {R\'epartition statistique des entiers sans grand facteur premier dans les progressions arithm\'etiques}, {\it Proc. London Math. Soc.} (3) {\bf 72} (1996), no. 3, p. 481--514.
 
\bibitem{H12a}
{\sc A. J. Harper},  On a paper of K. Soundararajan on smooth numbers in arithmetic progressions, {\it  J. Number Theory} {\bf  132} (2012), no. 1, 182--199.

\bibitem{H12} {\sc A. J. Harper}, {Bombieri-Vinogradov and Barban-Davenport-Halberstam type theorems for smooth numbers}, pr\'e-publication (2012).


\bibitem{H84} {\sc A. Hildebrand,} Integers free of large prime factors and the Riemann hypothesis, {\it  Mathematika} {\bf 31} (1984), no. 2, (1985), 258--271.

\bibitem{H85} {\sc A. Hildebrand,} Integers free of large prime divisors in short intervals, {\it Quart. J. Math. Oxford Ser.} (2) {\bf 36} (1985), no. 141, 57--69.


\bibitem{H86} {\sc A. Hildebrand,}
On the number of positive integers $\leq x$ and free of prime factors $> y$,  {\it J. Number
Theory} {\bf 22} (1986), 289--307.


\bibitem{HT86} {\sc A. Hildebrand  \& G. Tenenbaum,}
On integers free of large prime factors, {\it Trans. Amer. Math. Soc.}
{\bf 296} (1986), 265--290.


\bibitem{LT14}
{\sc A. Lachand \& G. Tenenbaum,}
Note sur les valeurs moyennes cribl\'ees de certaines fonctions arithm\'etiques, {\it Quart. J. Math.} (Oxford), (2014).

\bibitem{S89} {\sc E. Saias},   Sur le nombre des entiers sans grand facteur premier, {\it J. Number Theory} {\bf  32} (1989), no. 1, 78--99.

\bibitem{S08} {\sc K. Soundararajan}, The distribution of smooth numbers in arithmetic progressions, in: Anatomy of Integers, in: CRM Proc. Lect. Notes, vol. {\bf 46}, {\it Amer. Math. Soc.}, Providence, RI, 2008, pp. 115--128.


\bibitem{T08}
{\sc G. Tenenbaum,} {\it Introduction \`a 
la th\'eorie analytique et probabiliste des
nombres}, troisi\`eme \'edition, coll. \'Echelles, Belin, 2008, 592 pp.\par

\bibitem{W73} {\sc
D. Wolke}, {\"Uber die mittlere Verteilung der Werte zahlentheoretischer Funktionen auf Restklassen. I}, {\it Math. Ann.} {\bf 202} (1973), p. 1--25.

\end{thebibliography}
\end{document}